\documentclass[a4paper,12pt]{article}
\usepackage[utf8]{inputenc}

\usepackage{mathrsfs}
\usepackage{graphicx}
\usepackage{enumerate}
\usepackage{multicol}
\usepackage{color}
\usepackage{amsmath,amssymb,amscd}
\usepackage{amsthm}
\usepackage{mathabx}

\usepackage{pdfpages}
\usepackage{hyperref}

\usepackage{times}
\usepackage[varg]{txfonts}

\usepackage[top=1in, bottom=1in, left=0.5in, right=0.5in]{geometry}

\theoremstyle{plain}
\newtheorem{theorem}{Theorem}
\newtheorem{lemma}{Lemma}
\newtheorem{cor}{Corollary}
\newtheorem{proposition}{Proposition}
\newtheorem{conj}{Conjecture}

\theoremstyle{definition}
\newtheorem{definition}{Definition}

\theoremstyle{remark}
\newtheorem{remark}{Remark}

\newtheorem{opq}{Open Question}
\newtheorem*{ackn}{Acknowledgment}

\newcommand{\la}{\lambda}
\newcommand{\tla}{\widetilde{\lambda}}

\makeatletter
\newcommand{\subjclass}[2][1991]{%
  \let\@oldtitle\@title%
  \gdef\@title{\@oldtitle\footnotetext{#1 \emph{MSC classes.} #2}}%
}
\newcommand{\keywords}[1]{%
  \let\@@oldtitle\@title%
  \gdef\@title{\@@oldtitle\footnotetext{\emph{Keywords:} #1.}}%
}
\makeatother

\title{Polynomial equations for additive functions I. \\ The inner parameter case}
\author{Eszter Gselmann\thanks{University of Debrecen, e-mail: gselmann@science.unideb.hu}~ and Gergely Kiss\thanks{Alfr\'ed R\'enyi Institute of Mathematics, e-mail: kigergo57@gmail.com}}

\begin{document}
%\nocite{*}


\maketitle

\begin{abstract}
The aim of this sequence of work is to investigate polynomial equations satisfied by
additive functions. As a result of this, new characterization theorems for homomorphisms and derivations can
be given. More exactly, in this paper the following type of equation is considered
\[
 \sum_{i=1}^{n}f_{i}(x^{p_{i}})g_{i}(x^{q_{i}})= 0
 \qquad
 \left(x\in \mathbb{F}\right),
\]
where $n$ is a positive integer, $\mathbb{F}\subset \mathbb{C}$ is a field,
$f_{i}, g_{i}\colon \mathbb{F}\to \mathbb{C}$ are additive functions and $p_i, q_i$ are positive integers for all $i=1, \ldots, n$.
\end{abstract}

\section{Introduction and preliminaries}

Equations satisfied by additive functions play an important role not only in the theory of commutative algebra, but also in the theory of functional equations.
It is an important and challenging question how special morphisms (such as homomorphisms and derivations) can be characterized among additive mappings in general.
In this paper classes of multivariable algebraic equations are introduced with appropriate solutions as field homomorphisms and derivations.

Concerning all the cases we consider here, the involved additive functions are defined on a field $\mathbb{F}\subset \mathbb{C}$ and have values in the complex field, therefore
we introduce the preliminaries in this setting.

We adopt the standard notations, that is, $\mathbb{N}$ and $\mathbb{C}$ denote the set of positive integers and the set of complex numbers, respectively.

Henceforth we assume $\mathbb{F}\subset \mathbb{C}$ to be a field.

\begin{definition}
We say that a function $f\colon \mathbb{F}\to \mathbb{C}$ is \emph{additive} if it fulfills the so-called
\emph{Cauchy functional equation}, that is,
\[
 f(x+y)= f(x)+f(y)
 \qquad
 \left(x, y\in \mathbb{F}\right).
\]
An additive function $d\colon \mathbb{F} \to \mathbb{C}$ is termed to be a \emph{derivation} (of order 1) if it also fulfills the
\emph{Leibniz equation}, i.e.,
\[
 d(xy)= d(x)y+xd(y)
 \qquad
 \left(x, y\in \mathbb{F}\right).
\]
An additive function $\varphi\colon\mathbb{F}\to \mathbb{C}$ is said to be a {\emph homomorphism} if it is multiplicative as well, in other words, besides additivity we also have
\[
 \varphi(xy)= \varphi(x)\varphi(y)
 \qquad
 \left(x, y\in \mathbb{F}\right).
\]
If $\mathbb{F}= \mathbb{C}$ and $\varphi$ is an isomorphism, then $\varphi$ is called a \emph{complex automorphism}.
\end{definition}

Certain well-known equations are especially important. For instance, the additive solutions of following equation on a ring $R$ with char${R}\ne 2$
\[
f(x^{2})=2xf(x)
\qquad
\left(x\in \mathbb{}\right)
\]
are derivations, under some assumptions, where the additive mapping $f$ acts. As an extension of such type of results in \cite{Eba15, Eba17, GseKisVin18} the additive solutions of the equations
\[
\sum_{i=0}^{n}x^{i}f_{n+1-i}(x^{n+1-i})=0
\qquad
\left(x\in R\right)
\]
and
\[
\sum_{i=0}^{n} f(x^{p_{i}})x^{q_{i}}=0
\qquad
\left(x\in \mathbb{F}\right)
\]
were described, here $R$ denotes a ring with char$(R)\ge n$, while $\mathbb{F}\subset \mathbb{C}$ is a field.
By a polynomial equation of additive functions we mean an equation of the form
\[P(f_1^{r_1}(x^{s_1}), \dots, f_n^{r_n}(x^{s_n}))=0,\] where $P\colon \mathbb{C}^n \to \mathbb{C}$ is a $n$-variable polynomial, $r_i$, $s_i$ are positive integers  and $f_i$ denote the unknown additive functions.
Without further restrictions (e.g., on the polynomial $P$ or on the parameters $r_i, s_i$), the above equation is unfortunately too general for its solutions to be fully determined.
Indeed, it is not too hard to specify a polynomial $P$ and parameters $r_i, s_i$ so that the above equation is satisfied by all additive functions. Therefore, in this series, we will focus on the following classes of equations.

\[
\begin{array}{rcl}
 \displaystyle\sum_{i=1}^{n}f_{i}(x^{p_{i}})g_{i}(x^{q_{i}})&=& 0, \\[2mm]
 \displaystyle\sum_{i=1}^{n}f_{i}(x^{p_{i}})g_{i}(x)^{q_{i}}&=& 0, \\[2mm]
 \displaystyle\sum_{i=1}^{n}f_{i}(x)^{p_{i}}g_{i}(x)^{q_{i}}&=& 0, \\[2mm]
 \end{array}
 \qquad
 \left(x\in \mathbb{F}\right)
\]

In this paper the most impressive equation, namely
\begin{equation}\label{Eq_inner}
 \sum_{i=1}^{n}f_{i}(x^{p_{i}})g_{i}(x^{q_{i}})= 0,
 \qquad
 \left(x\in \mathbb{F}\right)
\end{equation}
 will be studied from this list with a fruitful theoretical description. (This we call 'inner case', as the parameters $p_i, q_i$ are exponents of the variable $x$, so they act on the domain of the functions $f_i, g_i$, respectively.)   Under some natural conditions equation \eqref{Eq_inner}  is satisfied by compositions of (higher order) derivations and homomorphisms. The purpose of this paper is about the converse by showing proper characterizations for the solutions of \eqref{Eq_inner} in the class of additive functions.

\subsection*{Structure of the paper}

In section \ref{sec2} the most important notations, terminology and theoretical background is summarized. Concerning the notions of polynomials, generalized polynomials, exponentials and exponential polynomials, here we follow the monograph of L.~Székelyhidi \cite{Sze91}. Besides these notions, decomposable functions, introduced by E.~Shulman in \cite{Shu10}, will play a key role in the second section. We show that all solutions of equation \eqref{Eq_inner} are decomposable functions. After that a result of M.~Laczkovich will be used, who proved in \cite{Lac19}  that on unital commutative topological semigroups, decomposable mappings are generalized exponential polynomials. On fields these are closely related to higher order derivations that were introduced by \cite{Rei98} and \cite{UngRei98}.
There will also be cases, when we will restrict to finitely generated subfields of $\mathbb{F}$ since on them higher order derivations are differential operators. The latter concept is significantly easier to calculate. This makes it possible to determine the exact upper bound for their orders.

The main results of the paper can be found in the third section.
At first some elementary yet important lemmata serve to settle reasonable conditions for the parameters $p_i, q_i$ such as the Homogenization Principle (see Lemma \ref{lemma_homogenization}), which ensures that the parameters satisfy
\[
  p_{i}+q_{i}=N
  \qquad
  \left(i=1, \ldots, n\right).
 \]

Based on the remarks and the examples of subsection \ref{SS2.1}, we will provide characterization theorems for equation \eqref{Eq_inner} under the following conditions
\begin{enumerate}[{C}(i)]
 \item the positive integers $p_{1}, \ldots, p_{n}$ satisfy $p_1<\dots <p_{n}$;
 \item for all $i=1, \ldots, n$ we have $p_{i}+q_{i}=N$;
 \item for all $i, j \in \left\{ 1, \ldots, n\right\}$, $i\neq j$ we have $p_{i}\neq q_{j}$.
\end{enumerate}

Further, according to Lemma \ref{lem_equ_inner}, the solutions of the above functional equations are sufficient to determine `up to equivalence'. This is because the functions $f_{i}$ and $g_{i}$ fulfill equation \eqref{Eq_inner}, if and only if for any automorphism $\varphi\colon \mathbb{C}\to \mathbb{C}$, the functions $\varphi \circ f_{i}$ and
$\varphi \circ g_{i}$ also fulfill \eqref{Eq_inner}, $i=1,\dots, n$.

 Looking at equation \eqref{Eq_inner}, it yields a technical problem that there is only one independent variable in the equation. At the same time, the involved functions are assumed to be additive. Thus the {\it polarization formula} for multi-additive functions can be used in the {\it symmetrization method}, which allows us to enlarge the number of independent variables from one to $N$.

In Lemma \ref{lem_decop_inner} it is shown that the functions $f_{i}$ and $g_{i}$ satisfies the system of equations given by this method are decomposable functions thus generalized exponential polynomials on the group $\mathbb{F}^{\times}$.
Theorem \ref{theorem_inner} says that for any $i\in \{1, \ldots, n\}$, in the variety of the functions $f_{i}$ and $g_{i}$ there is exactly one exponential $m_{i}$. Focusing on the irreducible solutions, this means that all $f_i$ (resp. $g_i$) are of the form $P_i\cdot m$ (resp. $Q_i\cdot m$), where $P_i$ (and $Q_i$) are (generalized) polynomials and $m$ is a unique exponential function.

Translating the problem to higher order derivations there can be found a natural basis of compositions of derivations of order 1 by using moment generating functions. Applying the arithmetic of derivations we get a sharp upper bound, which is $n-1$ for the order of derivation solutions under some conditions, see Theorem \ref{theorem_inner_degree}. We close this section with the study of some important special cases. In Conjecture \ref{conj_inner} and Open Problem \ref{op1} we pose problems on the exact order of the (higher order) derivation solutions in different settings. %Here, the equations
%\[
%f(x^{p})f(x^{N-p})=g(x^{q})g(x^{N-q})
%\qquad
%\left(x\in \mathbb{F}\right),
%\]
%\[
%  f(x^{p})f(x^{N-p})=\kappa f(x^{q})f(x^{N-q})
%  \qquad
%  \left(x\in \mathbb{F}\right),
%\]
%and
%\[
%  f(x^{p})g(x^{N-p})= \kappa f(x^{q})g(x^{N-q})
 % \qquad
%  \left(x\in \mathbb{F}\right).
%\]
%are considered in details, respectively.

\section{Notation, terminology and theoretical background }\label{sec2}
\subsection{Polynomials and generalized polynomials}

\begin{definition}
 Let $G, S$ be commutative semigroups (written additively), $n\in \mathbb{N}$ and let $A\colon G^{n}\to S$ be a function.
 We say that $A$ is \emph{$n$-additive} if it is a homomorphism of $G$ into $S$ in each variable.
 If $n=1$ or $n=2$ then the function $A$ is simply termed to be \emph{additive}
 or \emph{biadditive}, respectively.
\end{definition}

The \emph{diagonalization} or \emph{trace} of an $n$-additive
function $A\colon G^{n}\to S$ is defined as
 \[
  A^{\ast}(x)=A\left(x, \ldots, x\right)
  \qquad
  \left(x\in G\right).
 \]
As a direct consequence of the definition each $n$-additive function
$A\colon G^{n}\to S$ satisfies
%\begin{equation*}
\begin{multline*}
 A(x_{1}, \ldots, x_{i-1}, kx_{i}, x_{i+1}, \ldots, x_n)=
 kA(x_{1}, \ldots, x_{i-1}, x_{i}, x_{i+1}, \ldots, x_{n})\\
 \left(x_{1}, \ldots, x_{n}\in G\right)
\end{multline*}
%\end{equation*}
for all $i=1, \ldots, n$, where $k\in \mathbb{N}$ is arbitrary. The
same identity holds for any $k\in \mathbb{Z}$ provided that $G$ and
$S$ are groups, and for $k\in \mathbb{Q}$, provided that $G$ and $S$
are linear spaces over the rationals. For the diagonalization of $A$
we have
\[
 A^{\ast}(kx)=k^{n}A^{\ast}(x)
 \qquad
 \left(x\in G\right).
\]

The above notion can also be extended for the case $n=0$ by letting
$G^{0}=G$ and by calling $0$-additive any constant function from $G$ to $S$.

One of the most important theoretical results concerning
multiadditive functions is the so-called \emph{Polarization
formula}, that briefly expresses that every $n$-additive symmetric
function is \emph{uniquely} determined by its diagonalization under
some conditions on the domain as well as on the range. Suppose that
$G$ is a commutative semigroup and $S$ is a commutative group. The
action of the {\emph{difference operator}} $\Delta$ on a function
$f\colon G\to S$ is defined by the formula
\[\Delta_y f(x)=f(x+y)-f(x)
\qquad
\left(x, y\in G\right). \]
Note that the addition in the argument of the function is the
operation of the semigroup $G$ and the subtraction means the inverse
of the operation of the group $S$.

\begin{theorem}[Polarization formula]\label{theorem_polarization}
 Suppose that $G$ is a commutative semigroup, $S$ is a commutative group, $n\in \mathbb{N}$.
 If $A\colon G^{n}\to S$ is a symmetric, $n$-additive function, then for all
 $x, y_{1}, \ldots, y_{m}\in G$ we have
 \[
  \Delta_{y_{1}, \ldots, y_{m}}A^{\ast}(x)=
  \left\{
  \begin{array}{rcl}
   0 & \text{ if} & m>n \\
   n!A(y_{1}, \ldots, y_{m}) & \text{ if}& m=n.
  \end{array}
  \right.
 \]

\end{theorem}

\begin{cor}
 Suppose that $G$ is a commutative semigroup, $S$ is a commutative group, $n\in \mathbb{N}$.
 If $A\colon G^{n}\to S$ is a symmetric, $n$-additive function, then for all $x, y\in G$
 \[
  \Delta^{n}_{y}A^{\ast}(x)=n!A^{\ast}(y).
\]
\end{cor}

\begin{lemma}
\label{mainfact}
  Let $n\in \mathbb{N}$ and suppose that the multiplication by $n!$ is surjective in the commutative semigroup $G$ or injective in the commutative group $S$. Then for any symmetric, $n$-additive function $A\colon G^{n}\to S$, $A^{\ast}\equiv 0$ implies that
  $A$ is identically zero, as well.
\end{lemma}

\begin{definition}
 Let $G$ and $S$ be commutative semigroups, a function $p\colon G\to S$ is called a \emph{generalized polynomial} from $G$ to $S$, if it has a representation as the sum of diagonalizations of symmetric multi-additive functions from $G$ to $S$. In other words, a function $p\colon G\to S$ is a
 generalized polynomial if and only if, it has a representation
 \[
  p= \sum_{k=0}^{n}A^{\ast}_{k},
 \]
where $n$ is a nonnegative integer and $A_{k}\colon G^{k}\to S$ is a symmetric, $k$-additive function for each
$k=0, 1, \ldots, n$. In this case we also say that $p$ is a generalized polynomial \emph{of degree at most $n$}.

Let $n$ be a nonnegative integer, functions $p_{n}\colon G\to S$ of the form
\[
 p_{n}= A_{n}^{\ast},
\]
where $A_{n}\colon G^{n}\to S$ is a symmetric and $n$-additive mapping, are the so-called \emph{generalized monomials of degree $n$}.
\end{definition}

In this subsection $(G, \cdot)$ is assumed to be  a commutative group (written multiplicatively).

%\begin{definition}
%{\it Polynomials} are elements of the algebra generated by additive
%functions over $G$. Namely, if $n$ is a positive integer,
%$P\colon\mathbb{C}^{n}\to \mathbb{C}$ is a (classical) complex
%polynomial in
% $n$ variables and $a_{k}\colon G\to \mathbb{C}\; (k=1, \ldots, n)$ are additive functions, then the function
% \[
 % x\longmapsto P(a_{1}(x), \ldots, a_{n}(x))
 %\]
%is a polynomial and, also conversely, every polynomial can be
%represented in such a form \GK{(see \cite[Theorem 2.6]{Sze91} and the following paragraph therein)}.
%We call complex polynomials of this form \emph{normal polynomials}.
%\end{definition}

\begin{definition}
\emph{Polynomials} are elements of the algebra generated by additive
functions over $G$. More exactly, a mapping $f\colon G\to \mathbb{C}$ is called a \emph{polynomial} if there is a positive integer $n$, there exists a
(classical) complex
polynomial $P\colon\mathbb{C}^{n}\to \mathbb{C}$ in
 $n$ variables and there are additive functions  $a_{k}\colon G\to \mathbb{C}\; (k=1, \ldots, n)$ such that
\[
  f(x)= P(a_{1}(x), \ldots, a_{n}(x))
  \qquad
  \left(x\in G\right).
 \]
\end{definition}

\begin{remark}
 We recall that the elements of $\mathbb{N}^{n}$ for any positive integer $n$ are called
 ($n$-dimensional) \emph{multi-indices}.
 Addition, multiplication and inequalities between multi-indices of the same dimension are defined component-wise.
 Further, we define $x^{\alpha}$ for any $n$-dimensional multi-index $\alpha$ and for any
 $x=(x_{1}, \ldots, x_{n})$ in $\mathbb{C}^{n}$ by
 \[
  x^{\alpha}=\prod_{i=1}^{n}x_{i}^{\alpha_{i}}
 \]
where we always adopt the convention $0^0=0$. We also use the
notation $\lvert\alpha\rvert= \alpha_1+\cdots+\alpha_n$. With
these notations any polynomial of degree at most $N$ on the
commutative semigroup $G$ has the form
\[
 p(x)= \sum_{\lvert\alpha\rvert\leq N}c_{\alpha}a(x)^{\alpha}
 \qquad
 \left(x\in G\right),
\]
where $c_{\alpha}\in \mathbb{C}$ and $a=(a_1, \dots, a_n) \colon
G\to \mathbb{C}^{n}$ is an additive function. Furthermore, the
\emph{homogeneous term of degree $k$} of $p$ is
\[
 \sum_{\lvert\alpha\rvert=k}c_{\alpha}a(x)^{\alpha} .
\]
\end{remark}

\begin{lemma}[Lemma 2.7 of \cite{Sze91}]\label{L_lin_dep}
 Let $G$ be a commutative group,
 $n$ be a positive integer and let
 \[
  a=\left(a_{1}, \ldots, a_{n}\right),
 \]
where $a_{1}, \ldots, a_{n}$ are linearly independent complex valued
additive functions defined on $G$. Then the monomials
$\left\{a^{\alpha}\right\}$ for different multi-indices are linearly
independent.
\end{lemma}

\begin{definition}
A function $m\colon G\to \mathbb{C}$ is called an \emph{exponential}
function if it satisfies
\[
 m(xy)=m(x)m(y)
 \qquad
 \left(x,y\in G\right).
\]
Furthermore, on a(n)  \emph{(generalized) exponential polynomial} we mean a linear
combination of functions of the form $p \cdot m$, where $p$ is a
(generalized) polynomial and $m$ is an exponential function.
\end{definition}

%.
The following lemma shows that generalized exponential polynomial functions are linearly independent. Although it can be stated in a more general way (see \cite{Sze91}), we adopt it to our situation, when the functions are complex valued.
\begin{lemma}[Lemma 4.3 of \cite{Sze91}]\label{L_Lin_Ind}
 Let $G$ be a commutative group, $n$ a positive integer, % $K$ a field, $X$ a $K$-linear space,
 $m_{1}, \ldots, m_{n} \colon G\to \mathbb{C}\, (i=1, \ldots, n)$ be distinct nonzero exponentials and  $p_{1}, \ldots, p_{n} \colon  G\to \mathbb{K}\, (i=1, \ldots, n)$ be generalized polynomials. If  $\displaystyle \sum_{i=1}^n p_i\cdot m_i$ is identically zero, then for all $i=1, \ldots, n$ the generalized polynomial  $p_i$ is identically zero.
\end{lemma}

Additionally, we will need the analogous statement for polynomial expressions of generalized exponential polynomials which was proved in \cite{GseKisVin18}.

\begin{theorem}\label{T_Poly_Ind}
Let $\mathbb{K}$ be a field of characteristic $0$ and $k,l,N$ be
positive integers such that $k,l\le N$. Let $m_1, \dots, m_k\colon
\mathbb{K}^{\times}\to \mathbb{C}$ be distinct exponential functions
that are additive on $\mathbb{K}$, let $a_1, \dots, a_l\colon
\mathbb{K}^{\times}\to \mathbb{C}$ be additive functions that are
linearly independent over $\mathbb{C}$ and  for all $|s|\leq N$ let $ P_s\colon
\mathbb{C}^l\to \mathbb{C}$ be classical complex polynomials of $l$
variables. If
\[
    \sum_{|s|\le N } P_{s}(a_1, \dots, a_l) m_1^{s_1}\cdots m_k^{s_k}=0
\]
then for all $|s|\leq N$, the polynomials $P_s$ vanish identically.
\end{theorem}

\begin{definition}
 Let $G$ be a commutative group and $V\subseteq \mathbb{C}^G$ a set of functions. We say that $V$ is {\it translation invariant} if for every $f\in V$ the function $\tau_{g}f\in V$ also holds for all $g\in G$, where
 \[
  \tau_{g}f(h)= f(hg)
  \qquad
  \left(h\in G\right).
 \]
 \end{definition}

 In view of Theorem 10.1 of Székelyhidi \cite{Sze91}, any finite dimensional translation invariant linear
 space of complex valued functions on a commutative group consists of exponential polynomials.
 This implies that if $G$ is a commutative group, then any function
 $f\colon G\to \mathbb{C}$, satisfying the functional equation
 \[
  f(xy)= \sum_{i=1}^{n}g_{i}(x)h_{i}(y)
  \qquad
  \left(x, y\in G\right)
 \]
for some positive integer $n$ and functions $g_{i}, h_{i}\colon G\to \mathbb{C}$ ($i=1, \ldots, n$),
is an exponential polynomial of degree at most $n$.

This enlightens the connection between generalized polynomials and polynomials. It is easy to see that
each polynomial, that is, any function of the form
\[
  x\longmapsto P(a_{1}(x), \ldots, a_{n}(x)),
 \]
where $n$ is a positive integer,
$P\colon\mathbb{C}^{n}\to \mathbb{C}$ is a (classical) complex
polynomial in
$n$ variables and $a_{k}\colon G\to \mathbb{C}\; (k=1, \ldots, n)$ are additive functions, is a generalized polynomial.
The converse however is in general not true. A complex-valued generalized polynomial $p$ defined on a commutative group $G$ is a
polynomial \emph{if and only if} its variety (the linear space spanned by its translates) is of \emph{finite} dimension.

Henceforth, not only the notion of (exponential) polynomials, but also that of \emph{decomposable functions} will be used. The basics of this concept are due to
Shulman \cite{Shu10}, besides this we heavily rely on the work of Laczkovich \cite{Lac19}.

\begin{definition}
 Let $G$ be a group and $n\in \mathbb{N}, n\geq 2$.
 A function $F\colon G^{n}\to \mathbb{C}$ is said to be
 \emph{decomposable} if it can be written as a finite sum of products
 $F_{1}\cdots F_{k}$, where all $F_{i}$ depend on disjoint sets of variables.
\end{definition}

\begin{remark}
 Without loss of generality we can suppose that $k=2$ in the above definition, that is,
 decomposable functions are those mappings that can be written in the form
 \[
  F(x_{1}, \ldots, x_{n})= \sum_{E}\sum_{j}A_{j}^{E}B_{j}^{E}
 \]
where $E$ runs through all non-void proper subsets of $\left\{1,
\ldots, n\right\}$ and for each $E$ and $j$ the function $A_{j}^{E}$
depends only on variables $x_{i}$ with $i\in E$, while $B_{j}^{E}$
depends only on the variables $x_{i}$ with $i\notin E$.
\end{remark}

\begin{theorem}
Let $G$ be a commutative topological  semigroup with unit.
A continuous function $f\colon G\to \mathbb{C}$ is a generalized exponential polynomial
\emph{if and only if} there is a positive integer $n\geq 2$ such that the mapping
\[
G^{n} \ni (x_{1}, \ldots, x_{n}) \longmapsto f(x_1+\cdots+ x_n)
\]
is decomposable.
\end{theorem}

The notion of derivations can be extended in several ways. We will employ the concept of higher order derivations according to Reich \cite{Rei98} and Unger--Reich \cite{UngRei98}. For further results on characterization theorems on higher order derivations consult e.g. \cite{Eba15, Eba17, EbaRieSah} and
\cite{GseKisVin18}.

\begin{definition}
 Let $\mathbb{F}\subset \mathbb{C}$ be a field. The identically zero map is the only \emph{derivation of order zero}. For each $n\in \mathbb{N}$, an additive mapping
 $f\colon \mathbb{F}\to \mathbb{C}$ is termed to be a \emph{derivation of order $n$}, if there exists $B\colon \mathbb{F}\times \mathbb{F}\to \mathbb{C}$ such that
 $B$ is a bi-derivation of order $n-1$ (that is, $B$ is a derivation of order $n-1$ in each variable) and
 \[
  f(xy)-xf(y)-f(x)y=B(x, y)
  \qquad
  \left(x, y\in \mathbb{F}\right).
 \]
 The set of derivations of order $n$ of the ring $R$ will be denoted by $\mathscr{D}_{n}(\mathbb{F})$.
\end{definition}

\begin{remark}
\label{pathologic}
Since $\mathscr{D}_{0}(\mathbb{F})=\left\{0\right\}$, the only bi-derivation of order zero is the identically zero function, thus $f\in \mathscr{D}_{1}(\mathbb{F})$ if and only if
  \[
   f(xy)=xf(y)+f(x)y
   \qquad
   \left(x, y\in \mathbb{F}\right),
  \]
that is, the notions of first order derivations and derivations coincide. On the other hand for any $n\in \mathbb{N}$ the set $\mathscr{D}_{n}(\mathbb{F})\setminus \mathscr{D}_{n-1}(\mathbb{F})$ is nonempty because  $d_{1}\circ \cdots \circ d_{n}\in \mathscr{D}_{n}(\mathbb{F})$, but $d_{1}\circ \cdots \circ d_{n}\notin \mathscr{D}_{n-1}(R)$, where $d_{1}, \ldots, d_{n}\in \mathscr{D}_{1}(\mathbb{F})$ are non-identically zero derivations.
\end{remark}

For our future purposes the notion of differential operators will also be important, see \cite{KisLac18}.

\begin{definition}
 Let $\mathbb{F}\subset \mathbb{C}$ be a field. We say that the map
 $D \colon \mathbb{F}\to \mathbb{C}$ is a \emph{differential operator of order at most $n$} if $D$ is the linear combination, with coefficients from $\mathbb{F}$, of finitely many maps of the form
 $d_1 \circ \cdots \circ d_k$, where $d_1, \ldots, d_k$ are derivations on $\mathbb{F}$ and $k\leq n$. If $k = 0$ then we interpret $d_1\circ \cdots \circ d_k$ as the identity function.
 We denote by $\mathscr{O}_n(\mathbb{F})$ the set of differential operators of order at most $n$ defined on $\mathbb{F}$. We say that the order of a differential operator $D$ is
$n$ if $D \in \mathscr{O}_{n}(\mathbb{F})\setminus\mathscr{O}_{n-1}(\mathbb{F})$ (where $\mathscr{O}_{-1}(\mathbb{F})= \emptyset$, by definition).
\end{definition}

\begin{remark}
The term \emph{differential operator} is justified by the following fact. Let $\mathbb{K} =\mathbb{Q}(t_1, \ldots, t_k)$, where $t_1, \ldots, t_k$ are algebraically independent over $\mathbb{Q}$. Then $\mathbb{K}$ is the field of all rational functions of $t_1, \ldots, t_k$ with rational coefficients. It is clear that
\[
 d_i = \frac{\partial}{ \partial t_{i}}
 \]
is a derivation on $\mathbb{K}$ for every $i = 1, \ldots, k$. Therefore, every
differential operator
\[
 D= \sum_{i_{1}+\cdots+i_{k}\leq n}c_{i_{1}, \ldots, i_{k}}\cdot \frac{\partial^{i_{1}+\cdots+i_{k}}}{\partial t_{1}^{i_{1}}\cdots \partial t_{k}^{i_{k}}},
\]
where the coefficients $c_{i_{1}, \ldots, i_{k}}$ belong to $\mathbb{K}$, is a differential operator of order at most $n$, and also conversely, if $D$ is a differential operator of order at most $n$ on the field $\mathbb{K} = \mathbb{Q}(t_1, \ldots, t_k)$, then $D$ is of the above form.
\end{remark}

The main result of \cite{KisLac18} is Theorem 1.1 that reads in our settings as follows.

\begin{theorem}\label{theorem_KisLac}
 Let $\mathbb{F}\subset \mathbb{C}$ be a field and let $n$ be a positive integer. Then, for every function $D \colon \mathbb{F}\to \mathbb{C}$, the
following are equivalent.
\begin{enumerate}[(i)]
\item $D\in \mathscr{D}_{n}(\mathbb{F})$
\item $D\in \mathrm{cl}\left(\mathscr{O}_{n}(\mathbb{F})\right)$
\item $D$ is additive on $\mathbb{F}$, $D(1) = 0$, and $D/j$, as a map from the group $\mathbb{F}^{\times}$ to $\mathbb{C}$, is a generalized polynomial of degree at most $n$. Here $j$ stands for the identity map defined on $\mathbb{F}$.
\end{enumerate}
\end{theorem}

\section{Results}

\subsection{Elementary observations: reduction of the problem}\label{SS2.1}

This part begins with some elementary, yet fundamental observations. As the following lemmata show, the original problem can be reduced to a more simpler equation.

\begin{lemma}[Homogenization]\label{lemma_homogenization}
 Let $n$ be a positive integer, $\mathbb{F}\subset \mathbb{C}$ be a field and
 $p_{1}, \ldots, p_{n}, q_{1}, \ldots, q_{n}$ be fixed positive integers.
 Assume that the additive functions $f_{1}, \ldots, f_{n}, g_{1}, \ldots, g_{n}\colon \mathbb{F}\to \mathbb{C}$ satisfy functional equation \eqref{Eq_inner}, that is,
 \[
\sum_{i=1}^{n}f_{i}(x^{p_{i}})g_{i}(x^{q_{i}})= 0
 \]
 for each $x\in \mathbb{F}$.
 If the set $\left\{ p_{1}, \ldots, p_{n}\right\}$ has a partition $\mathcal{P}_{1}, \ldots, \mathcal{P}_{k}$ with the property
 \[
 \text{if } p_{\alpha}, p_{\beta} \in \mathcal{P}_{j} \text{ for a certain index $j$, then } p_{\alpha}+q_{\alpha}= p_{\beta}+q_{\beta},
 \]
then the system of equations
\[
 \sum_{p_{\alpha}\in \mathcal{P}_{j}} f_{\alpha}(x^{p_{\alpha}})g_{\alpha}(x^{q_{\alpha}})=0
 \qquad
 \left(x\in \mathbb{F}, j=1, \ldots, k\right)
\]
is satisfied.
\end{lemma}

\begin{proof}
 Let $n$ be a positive integer, $\mathbb{F}\subset \mathbb{C}$ be a field and
 $p_{1}, \ldots, p_{n}, q_{1}, \ldots, q_{n}$ be fixed positive integers.
 Assume that the additive functions $f_{1}, \ldots, f_{n}, g_{1}, \allowbreak \ldots, g_{n}\colon \mathbb{F}\to \mathbb{C}$ satisfy functional equation \eqref{Eq_inner}
 for each $x\in \mathbb{F}$.
 Assume further that the set $\left\{ p_{1}, \ldots, p_{n}\right\}$ has a partition $\mathcal{P}_{1}, \ldots, \mathcal{P}_{k}$ with the property
 \[
 \text{if } p_{\alpha}, p_{\beta} \in \mathcal{P}_{j} \text{ for a certain index $j$, then } p_{\alpha}+q_{\alpha}= p_{\beta}+q_{\beta}.
 \]
 Observe that for all $i=1, \ldots, n$, the mapping
 \[
  \mathbb{F}\ni x\longmapsto f_{i}(x^{p_{i}})g_{i}(x^{q_{i}})
 \]
is a generalized monomial of degree $p_{i}+q_{i}$. Indeed, it is the diagonalization of the symmetric
$(p_{i}+q_{i})$-additive mapping
\[
 \mathbb{F}^{p_{i}+q_{i}} \ni (x_{1}, \ldots, x_{p_{i}+q_{i}})
 \longmapsto
 f_{i}(x_{\sigma(1)}\cdots x_{\sigma(p_{i})})g_{i}(x_{\sigma(p_{i}+1)}\cdots x_{\sigma(p_{i}+q_{i})}).
\]
Since $\mathbb{F}\subset \mathbb{C}$, we necessarily have $\mathbb{Q}\subset \mathbb{F}$.
Let now $r\in \mathbb{Q}$ be arbitrary and substitute $rx$ in place of $x$ in equation \eqref{Eq_inner} to get
\[
 \sum_{i=1}^{n}f_{i}((rx)^{p_{i}})g_{i}((rx)^{q_{i}})= 0
 \qquad
 \left(r\in \mathbb{Q}, x\in \mathbb{F}\right).
\]
Using the $\mathbb{Q}$-homogeneity of the additive functions $f_{1}, \ldots, f_{n}$ and $g_{1}, \ldots, g_{n}$, we deduce
\begin{multline*}
 0=
 \sum_{i=1}^{n}f_{i}((rx)^{p_{i}})g_{i}((rx)^{q_{i}})=
 \sum_{i=1}^{n}f_{i}(r^{p_{i}}x^{p_{i}})g_{i}(r^{q_{i}}x^{q_{i}})
 \\
 =
 \sum_{i=1}^{n}r^{p_{i}+q_{i}}f_{i}(x^{p_{i}})g_{i}(x^{q_{i}})
 =
 \sum_{j=1}^{k} \sum_{p_{\alpha}\in \mathcal{P}_{j}}r^{p_{\alpha}+q_{\alpha}}  f_{\alpha}(x^{p_{\alpha}})g_{\alpha}(x^{q_{\alpha}})
 \\
 \left(r\in \mathbb
 Q, x\in \mathbb{F}\right).
\end{multline*}
Note that the right hand side of this equation is a (classical) polynomial in $r$ which is identically zero. Thus all of its coefficients should be (identically) zero, yielding that the system of equations
\[
 \sum_{p_{\alpha}\in \mathcal{P}_{j}} f_{\alpha}(x^{p_{\alpha}})g_{\alpha}(x^{q_{\alpha}})=0
 \qquad
 \left(x\in \mathbb{F}, j=1, \ldots, k\right)
\]
is fulfilled.
\end{proof}

\begin{remark}\label{rem2.1}
 The above lemma guarantees that \emph{ab initio}
 \[
  p_{i}+q_{i}=N
  \qquad
  \left(i=1, \ldots, n\right)
 \]
can be assumed. Otherwise, after using the above homogenization, we get a system of functional equations in which this condition is already fulfilled.
For instance, due to the above lemma, if the additive functions $f_{1}, \ldots, f_{5}\colon \mathbb{F}\to \mathbb{C}$ and $g_{1}, \ldots, g_{5}\colon \mathbb{F}\to \mathbb{C}$ satisfy equation
\begin{multline*}
 f_{1}(x^{24})g_{1}(x^{5})+f_{2}(x^{20})g_{2}(x^{9})+f_{3}(x^{19})g_{3}(x^{10})
 \\
 +
 f_{4}(x^{13})g_{4}(x^{7})+
 f_{5}(x^{12})g_{4}(x^{8})
 =
 0
 \qquad
 \left(x\in \mathbb{F}\right)
\end{multline*}
then
the equations
\[
 f_{1}(x^{24})g_{1}(x^{5})+f_{2}(x^{20})g_{2}(x^{9})+f_{3}(x^{19})g_{3}(x^{10})=0
 \qquad
 \left(x\in \mathbb{F}\right)
 \]
and
\[
   f_{4}(x^{13})g_{4}(x^{7})+ f_{5}(x^{12})g_{4}(x^{8})=0
   \qquad
   \left(x\in \mathbb{F}\right)
\]
are also fulfilled (separately).
\end{remark}

\begin{remark}\label{rem2.3}
 At first glance the assumption that $p_{1}, \ldots, p_{n}$ are different seems a reasonable and sufficient supposition.
 Clearly, if the parameters are not necessarily different then we cannot expect anything special for the form of the involved additive functions.
 Indeed, let $L\subset \mathbb{C}^{n}$ be a linear subspace and let
 $f_{1}, \ldots, f_{n}\colon \mathbb{F}\to \mathbb{C}$ and $g_{1}, \ldots, g_{n}\colon \mathbb{F}\to \mathbb{C}$ be additive functions such that
 $\mathrm{rng}(f)\subset L$ and $\mathrm{rng}(g)\subset L^{\perp}$, where
 \[
  f(x)= \left(f_{1}(x), \ldots, f_{n}(x)\right)
  \quad
  \text{and}
  \quad
  g(x)= \left(g_{1}(x), \ldots, g_{n}(x)\right)
  \quad
  \left(x\in \mathbb{F}\right).
 \]
In this case
\[
 \sum_{i=1}^{n}f_{i}(x)g_{i}(x)= \langle f(x), g(x) \rangle = 0
 \qquad
 \left(x\in \mathbb{F}\right).
\]
This shows the necessity of the above assumption.
Unfortunately, the sufficiency fails to hold.
 To see this, let $p$ and $q$ be positive integers and $f\colon \mathbb{F}\to \mathbb{C}$ be an \emph{arbitrary} additive function and
 define the complex-valued functions $f_{1}, g_{1}, f_{2}, g_{2}$ on $\mathbb{F}$ by
 \[
  f_{1}(x)= f(x) \quad
  g_{1}(x)= f(x) \quad
  f_{2}(x)= if(x) \quad
  g_{2}(x)= if(x)
  \qquad
  \left(x\in \mathbb{F}\right).
 \]
An immediate computation shows that we have
\[
 f_{1}(x^{p})g_{1}(x^{q})+f_{2}(x^{p})g_{2}(x^{q})= 0
 \qquad
 \left(x\in \mathbb{F}\right).
\]
\end{remark}

In view of the above remarks, from now on, the following assumptions are adopted.

\begin{enumerate}%[{C}(i)]
 \item[C(i)] the positive integers $p_{1}, \ldots, p_{n}$ are arranged in a strictly increasing order, i.e., $p_1<\dots <p_{n}$;
 \item[C(ii)] for all $i=1, \ldots, n$ we have $p_{i}+q_{i}=N$;
 \item[C(iii)] for all $i, j \in \left\{ 1, \ldots, n\right\}$, $i\neq j$ we have $p_{i}\neq q_{j}$.
\end{enumerate}

\begin{remark}\label{rem2.4}
 Define the relation $\sim$ on $\mathbb{F}^{\mathbb{C}}$ by
 $  f\sim g $ if and only if there exists an automorphism $\varphi \colon \mathbb{C}\to \mathbb{C}$
  such that $\varphi \circ f=g$.
Obviously $\sim$ is an equivalence relation on $\mathbb{F}^{\mathbb{C}}$ that induces a partition on $\mathbb{F}^{\mathbb{C}}$.
\end{remark}

\begin{lemma}[Equivalence]\label{lem_equ_inner}
 Let $n$ be a positive integer, $\mathbb{F}\subset \mathbb{C}$ be a field and
 $p_{1}, \ldots, p_{n}, q_{1}, \ldots, q_{n}$ be fixed positive integers fulfilling the conditions C(i)-C(iii) of Remark \ref{rem2.3}.
 Assume that the additive functions $f_{1}, \ldots, f_{n},$ $g_{1}, \ldots, g_{n}\colon \mathbb{F}\to \mathbb{C}$ satisfy functional equation
 \eqref{Eq_inner}. Then for an arbitrary automorphism $\varphi \colon \mathbb{C}\to \mathbb{C}$ the functions
 $\varphi \circ f_{1}, \ldots, \varphi \circ f_{n}, \varphi \circ g_{1}, \ldots, \varphi \circ g_{n}$ also fulfill equation \eqref{Eq_inner}.
\end{lemma}

\subsection{Structure of solutions}
 We can always restrict ourselves to the case when all the involved functions are non-identically zero. Otherwise, the number of the terms appearing in equation \eqref{Eq_inner} can be reduced.

\begin{lemma}[Symmetrization]\label{lem_monom}
 Let $k$ and $n$ be positive integers, $\mathbb{F}\subset \mathbb{C}$ be a field and
 $m_{1}, \ldots, m_{n}\colon \mathbb{F}\to \mathbb{C}$ be monomials of degree $k$.
 If
 \[
  \sum_{i=1}^{n}m_{i}(x)=0
 \]
holds for all $x\in \mathbb{F}$, then
\[
 \sum_{i=1}^{n}M_{i}(x_{1}, \ldots, x_{k})=0
\]
is fulfilled for all $x_{1}, \ldots, x_{k}$, where for all $i=1, \ldots, n$, the mapping
$M_{i}\colon \mathbb{F}^{k}\to \mathbb{C}$ is the uniquely determined symmetric, $k$-additive function such that
\[
 M_{i}(x, \ldots, x)= m_{i}(x)
 \qquad
 \left(x\in \mathbb{F}\right).
\]
\end{lemma}

\begin{proof}
 Let $k$ and $n$ be positive integers, $\mathbb{F}\subset \mathbb{C}$ be a field and
 $m_{1}, \ldots, m_{n}\colon \mathbb{F}\allowbreak\to \mathbb{C}$ be monomials of degree $k$ and
 assume that
 \[
  \sum_{i=1}^{n}m_{i}(x)=0
 \]
holds for all $x\in \mathbb{F}$. Since for all $i=1, \ldots, n$, the function $m_{i}$ is a monomial of degree $k$, there exists a symmetric, $k$-additive function $M_{i}\colon \mathbb{F}^{k}\to \mathbb{C}$ such that we have
\[
 M_{i}(x, \ldots, x)= m_{i}(x)
 \qquad
 \left(x\in \mathbb{F}\right).
\]
Obviously the mapping $\displaystyle\sum_{i=1}^{n}m_{i}$ is a monomial of degree $k$ which is, by the assumptions, identically zero. To this monomial there also corresponds a symmetric and $k$-additive mapping, namely
\[
\mathbb{F}^{k}\ni (x_{1}, \ldots, x_{k}) \longmapsto \sum_{i=1}^{n}M_{i}(x_{1}, \ldots x_{k}).
\]
Observe that the trace of this symmetric and $k$-additive mapping is identically zero. At the same time, due to the Polarization formula (Theorem \ref{theorem_polarization}), every symmetric and $k$-additive function is \emph{uniquely} determined by its trace. Thus
\[
 \sum_{i=1}^{n}M_{i}(x_{1}, \ldots, x_{k})=0
\]
for all $x_{1}, \ldots, x_{k}\in \mathbb{F}$.
\end{proof}

\begin{lemma}[Symmetrization]\label{lem_sym_inner}
 Let $n$ be a positive integer, $\mathbb{F}\subset \mathbb{C}$ be a field and
 $p_{1}, \ldots, p_{n}, q_{1}, \ldots, q_{n}$ be fixed positive integers fulfilling conditions C(ii), i.e., there is a $N\in \mathbb{N}$ such $p_i+q_i=N$ for all $i=1,\dots,n$.
 Assume that the additive functions $f_{1}, \ldots, f_{n}, g_{1}, \ldots, g_{n}\colon \mathbb{F}\to \mathbb{C}$ satisfy functional equation
 \eqref{Eq_inner}
 for each $x\in \mathbb{F}$. Then
 \[
  \frac{1}{N!} \sum_{\sigma \in \mathscr{S}_{N}}\sum_{i=1}^{n} f_{i}(x_{\sigma(1)} \cdots x_{\sigma(p_{i})}) \cdot g_{i}(x_{\sigma(p_{i}+1)} \cdots x_{\sigma(N)})=0
 \]
holds for all $x_{1}, \ldots, x_{N}\in \mathbb{F}$.
\end{lemma}

\begin{proof}
 Let $n$ be a positive integer, $\mathbb{F}\subset \mathbb{C}$ be a field and
 $p_{1}, \ldots, p_{n}, q_{1}, \ldots, q_{n}$ be fixed positive integers fulfilling conditions C(ii).
 Assume that the additive functions $f_{1}, \ldots, f_{n}, g_{1}, \ldots, g_{n}\colon \mathbb{F}\to \mathbb{C}$ satisfy functional equation
 \[
  \sum_{i=1}^{n}f_{i}(x^{p_{i}})g_{i}(x^{q_{i}})= 0
 \]
 for each $x\in \mathbb{F}$. Due to the additivity of the functions
 $f_{1}, \ldots, f_{n}$ and $g_{1}, \ldots, g_{n}$ for all $i=1, \ldots, n$, the mapping
 \[
  x\longmapsto f_{i}(x^{p_{i}})g_{i}(x^{q_{i}})
 \]
 is a monomial of degree $p_{i}+q_{i}=N$. Further, it is the trace of the symmetric and $N$-additive mapping
\begin{multline*}
F_{i}(x_{1}, \ldots, x_{N})
\\
= \frac{1}{N!} \sum_{\sigma \in \mathscr{S}_{N}} f_{i}(x_{\sigma(1)} \cdots x_{\sigma(p_{i})}) \cdot g_{i}(x_{\sigma(p_{i}+1)} \cdots x_{\sigma(N)})
 \\
 \left(x_{1}, \ldots, x_{N}\in \mathbb{F}\right).
\end{multline*}
Therefore, the statement follows from Lemma \ref{lem_monom}.
\end{proof}

\subsection{Solutions of equation \texorpdfstring{\eqref{Eq_inner}}{}}\label{SS2.2}

The main purpose of the subsection is to describe under the conditions C(i)--C(iii), the solution space of equation \eqref{Eq_inner}. We first prove that solutions of equation \eqref{Eq_inner} are decomposable functions on the multiplicative group $\mathbb{F}^{\times}$.  In view of Laczkovich \cite{Lac19}, this immediately yields that the solutions of equation \eqref{Eq_inner} are generalized exponential polynomials of this group.

\begin{lemma}\label{lem_decop_inner}
 Let $n$ be a positive integer, $\mathbb{F}\subset \mathbb{C}$ be a field and
 $p_{1}, \ldots, p_{n}, q_{1}, \ldots, \allowbreak q_{n}$ be fixed positive integers fulfilling conditions C(i) and C(ii).
 Assume that the additive functions $f_{1}, \ldots, f_{n}, g_{1}, \ldots, g_{n}\colon \mathbb{F}\to \mathbb{C}$ satisfy functional equation
 \eqref{Eq_inner}
 for each $x\in \mathbb{F}$. Then all the functions $f_{1}, \ldots, f_{n}$ as well as $g_{1}, \ldots, g_{n}$ are decomposable functions of the group $\mathbb{F}^{\times}$.
\end{lemma}

\begin{proof}
 Let $n$ be a positive integer, $\mathbb{F}\subset \mathbb{C}$ be a field and
 $p_{1}, \ldots, p_{n}, q_{1}, \ldots, q_{n}$ be fixed positive integers fulfilling conditions C(i) and C(ii).

 Let us assume first that condition C(iii) is also satisfied.
 Assume that the additive functions $f_{1}, \ldots, f_{n},$ $g_{1}, \ldots, \allowbreak g_{n}\colon \mathbb{F}\to \mathbb{C}$ satisfy functional equation
 \eqref{Eq_inner}
 for each $x\in \mathbb{F}$.
 Let
 \[
  S= \left\{p_{1}, \ldots, p_{n} \right\} \cup \left\{ q_{1}, \ldots, q_{n}\right\}.
 \]
Then $\max S= \max \left\{ p_{n}, q_{1} \right\}$. By condition C(iii), we have $p_{n}\neq q_{1}$.
 Without the loss of generality $p_{n}>q_{1}$ can be assumed, otherwise we follow a similar argument.
 In view of  Lemma \ref{lem_sym_inner}, we have
 \[
  \frac{1}{N!} \sum_{\sigma \in \mathscr{S}_{N}}\sum_{i=1}^{n} f_{i}(x_{\sigma(1)} \cdots x_{\sigma(p_{i})}) \cdot g_{i}(x_{\sigma(p_{i}+1)} \cdots x_{\sigma(N)})=0
 \]
for all $x_{1}, \ldots, x_{N}\in \mathbb{F}$, or after some rearrangement,
\begin{multline*}
   \frac{1}{N!} \sum_{\sigma \in \mathscr{S}_{N}} f_{n}(x_{\sigma(1)} \cdots x_{\sigma(p_{n})}) \cdot g_{n}(x_{\sigma(p_{n}+1)} \cdots x_{\sigma(N)})
   \\=
   -\frac{1}{N!} \sum_{\sigma \in \mathscr{S}_{N}}\sum_{i=1}^{n-1} f_{i}(x_{\sigma(1)} \cdots x_{\sigma(p_{i})}) \cdot g_{i}(x_{\sigma(p_{i}+1)} \cdots x_{\sigma(N)})
   \\
   \left(x_{1}, \ldots, x_{N}\in \mathbb{F}^{\times}\right).
   \end{multline*}
Let now
\[
 x_{p_{n}+1}= \cdots = x_{N}= 1,
\]
then the above identity says that $g_{n}(1) \cdot f_{n}$ is decomposable. If $g_{n}(1)$ were zero, but $g_{n}$ would not be identically zero, then there would exist
$a\in \mathbb{F}^{\times}$ such that
$g_{n}(a)\neq 0$. In this case the above substitutions should be modified to
\[
 x_{p_{n}+1}= a , \; x_{p_{n}+2}= \cdots = x_{N}= 1,
\]
to get the same conclusion.

If $p_{n}= q_{j}$ for some $j=\{1, \dots, n\}$ (i.e., C(iii) does not hold), then without loss of generality we may assume that $j=1$, otherwise we may change the role of $f_i$ and $g_i$, and $p_i$ and $q_i$, respectively, and proceed as above.
If $p_{n}= q_{1}$, then we have
\begin{multline*}
   \frac{1}{N!} \sum_{\sigma \in \mathscr{S}_{n}} f_{n}(x_{\sigma(1)} \cdots x_{\sigma(p_{n})}) \cdot g_{n}(x_{\sigma(p_{n}+1)} \cdots x_{\sigma(N)})
   \\
   +
   \frac{1}{N!} \sum_{\sigma \in \mathscr{S}_{n}} g_{1}(x_{\sigma(1)} \cdots x_{\sigma(p_{n})}) \cdot f_{1}(x_{\sigma(p_{n}+1)} \cdots x_{\sigma(N)})
   \\=
   -\frac{1}{N!} \sum_{\sigma \in \mathscr{S}_{n}}\sum_{i=2}^{n-1} f_{i}(x_{\sigma(1)} \cdots x_{\sigma(p_{i})}) \cdot g_{i}(x_{\sigma(p_{i}+1)} \cdots x_{\sigma(N)})
   \\
   \left(x_{1}, \ldots, x_{N}\in \mathbb{F}^{\times}\right).
   \end{multline*}
This equation with the substitutions
\[
 x_{p_{n}+1}= \cdots = x_{N}= 1,
\]
yields that a linear combination of $f_{n}$ and $g_{1}$ is decomposable.
If $\left\{f_{n}, g_{1} \right\}$ is linearly dependent, then this obviously means that both $f_{n}$ and $g_{1}$ are decomposable functions.
If this system is not linearly dependent, then there exist
$a, b\in \mathbb{F}^{\times}$, $a\neq b$ and different complex constants $c_{1}$ and $c_{2}$ such that
\[
 f_{n}(a)= c_{1}g_{1}(a)
 \qquad
 f_{n}(b)= c_{2}g_{1}(b).
\]
With the substitutions
\[
 x_{p_{n}+1}= a , \; x_{p_{n}+2}= \cdots = x_{N}= 1,
\]
and
\[
 x_{p_{n}+1}= b , \; x_{p_{n}+2}= \cdots = x_{N}= 1,
\]
 we get that the functions
 \[
  (x_{1}, \ldots, x_{p_{1}}) \longmapsto g_{1}(x_{1} \cdots x_{p_{1}}) f_{n}(a) + f_{n}(x_{1} \cdots x_{p_{1}}) g_{1}(a)
 \]
and
 \[
  (x_{1}, \ldots, x_{p_{1}}) \longmapsto g_{1}(x_{1} \cdots x_{p_{1}}) f_{n}(b) + f_{n}(x_{1} \cdots x_{p_{1}}) g_{1}(b)
 \]
 are decomposable. Since finite linear combinations of decomposable functions are also decomposable, it follows that $f_{n}$ and $g_{1}$ are decomposable, separately.

 After that, let us consider the set $S\setminus \left\{ p_{n}\right\}$ and apply the above argument for this set.
 With this step-by-step descending argument the statement of the lemma follows.
 \end{proof}

 \begin{remark}
We emphasize that condition C(iii) in Lemma \ref{lem_decop_inner} has not been assumed. Thus, the fact that the additive solutions of \eqref{Eq_inner} are decomposable functions can be deduced only under the conditions C(i) and C(ii). On the other hand, for our purpose to describe the solutions more concretely we have to assume also condition C(iii) to avoid further difficulties. We believe however that most of our methods can work similarly only under the conditions C(i) and C(ii).
\end{remark}

\begin{remark}
  Note also that in general we cannot state more than that the involved functions
  $f_{1}, \ldots, f_{n}$ and $g_{1}, \ldots, g_{n}$ are decomposable functions on the commutative group $\mathbb{F}^{\times}$. In other words, we can only state that the solutions of the functional equation in question are higher order derivations (see below Corollary \ref{cor_inner}). In general it is not true that the solutions of this functional equation are differential operators (i.e., exponential polynomials of the multiplicative group $\mathbb{F}^{\times}$). To see this, let us consider the functional equation
  \[
   xf_{1}(x^{6})+x^{2}f_{2}(x^{5})+x^{3}f_{3}(x^{4})=0
   \qquad
   \left(x\in \mathbb{F}\right).
  \]
Indeed, using the results of \cite{GseKisVin18}, we deduce that $f_{1}, f_{2}, f_{3}\in \mathscr{D}_{2}(\mathbb{F})$.
\end{remark}

 Due to a result of Laczkovich \cite{Lac19} and under the assumptions of Lemma \ref{lem_decop_inner},
 there exist a positive integer $l$, there are generalized polynomials $P_{k, i}, Q_{k, i}\colon \allowbreak \mathbb{F}^{\times}\to \mathbb{C}$ ($k=1, \ldots, l$ and $i=1, \ldots, n$) and there exist linearly independent exponentials $m_{1}, \ldots, m_{l}\colon \mathbb{F}^{\times}\to \mathbb{C}$ such that
 \[
  f_{i}(x)= \sum_{k=1}^{l}P_{k, i}(x)m_{k}(x)
  \;
  \text{and}
  \;
  g_{i}(x)= \sum_{k=1}^{l}Q_{k, i}(x)m_{k}(x)
  \qquad
  \left(x\in \mathbb{F}^{\times}\right).
 \]

  Since the generalized polynomials on any finitely generated field are polynomials and the exponentials $m_{1}, \ldots, m_{k}$ are linearly independent, we can apply Theorem \ref{T_Poly_Ind}. This implies that % they are algebraically independent, too.
   after substituting the above form the functions $f_{1}, \ldots, f_{n}$ and $g_{1}, \ldots, g_{n}$ and using that for each $\kappa \in \mathbb{N}$ and $k=1, \ldots, n$, we have
  \[
   m_{k}(x^{\kappa})= m_{k}(x)^{\kappa}
   \qquad
   \left(x\in \mathbb{F}^{\times}\right),
  \]
 especially
 \[
  \sum_{i=1}^{n}P_{k, i}(x^{p_{i}})Q_{k, i}(x^{q_{i}})=0
 \qquad
 \left(x\in \mathbb{F}^{\times}\right)
 \]
 follows for all $k=1, \ldots, l$. This tells us that it is enough to solve equation \eqref{Eq_inner} for generalized polynomials of the group $\mathbb{F}^{\times}$. In fact, we can prove the following.

\begin{theorem}\label{theorem_inner}
 Let $n$ be a positive integer, $\mathbb{F}\subset \mathbb{C}$ be a field and
 $p_{1}, \ldots, p_{n}, q_{1}, \allowbreak \ldots, \allowbreak q_{n}$ be fixed positive integers fulfilling conditions C(i)--C(iii).
 Assume that the additive functions $f_{1}, \ldots, f_{n}, g_{1}, \ldots, g_{n}\colon \mathbb{F}\to \mathbb{C}$ satisfy functional equation
 \eqref{Eq_inner}
 for each $x\in \mathbb{F}$. Then there exists a positive integer $l$, there exist exponentials $m_i\colon \mathbb{F}^{\times}\to \mathbb{C}$ and there are generalized polynomials $P_{i}, Q_{i}\colon \mathbb{F}^{\times}\to \mathbb{C}$ of degree at most $l$ such that
 \[
  f_{i}(x)= P_{i}(x)m_i(x)
  \qquad
  \text{and}
  \qquad
  g_{i}(x)= Q_{i}(x)m_i(x)
  \qquad
  \left(x\in \mathbb{F}^{\times}\right)
 \]
for each $i=1, \ldots, n$.
\end{theorem}

\begin{proof}
 As we saw above, under the hypothesis of the lemma, if the functions $f_{1}, \ldots, f_{n}$ and $g_{1}, \ldots, g_{n}$ solve equation \eqref{Eq_inner}, then there exist a positive integer $l$, there are generalized polynomials $P_{k, i}, Q_{k, i}\colon \mathbb{F}^{\times}\to \mathbb{C}$ ($k=1, \ldots, l$ and $i=1, \ldots, n$) and there exist linearly independent exponentials $m_{1}, \ldots, m_{l}\colon \mathbb{F}^{\times}\to \mathbb{C}$ such that
 \[
  f_{i}(x)= \sum_{k=1}^{l}P_{k, i}(x)m_{k}(x)
  \;
  \text{and}
  \;
  g_{i}(x)= \sum_{k=1}^{l}Q_{k, i}(x)m_{k}(x)
  \qquad
  \left(x\in \mathbb{F}^{\times}\right).
 \]
Assume to the contrary that $l\geq 2$.

Let
 \[
  S= \left\{p_{1}, \ldots, p_{n} \right\} \cup \left\{ q_{1}, \ldots, q_{n}\right\}.
 \]
Then due to conditions C(i)--C(iii) we have $\max S= \max \left\{ p_{n}, q_{1} \right\}$. Similarly as in Lemma \ref{lem_decop_inner}, without the loss of generality $p_{n}>q_{1}$ can be assumed, otherwise we follow a similar argument.
By our assumption $l\geq 2$, this yields that there exist different exponential terms in $f_{1}$ and $g_{1}$ with nonzero polynomial coefficients. For the sake of simplicity, suppose that these different exponentials are $m_{1}$ and $m_{2}$. Since $\max S= p_{1}$, the term $m_{1}^{p_{1}}m_{2}^{q_{1}}$ appears only in $f_{1}(x^{p_{1}})g_{1}(x^{q_{1}})$ while expanding equation \eqref{Eq_inner}.
Since generalized polynomials $P_{k,i}$ and exponentials $m_k$ satisfies the conditions of Theorem \ref{T_Poly_Ind} for every finitely generated subfield of $\mathbb{F}$,
the coefficient of the above-mentioned term which is $P_{1, 1}(x^{p_{1}})Q_{1, 1}(x^{q_{1}})$ has to vanish on $\mathbb{F}$. From this we can deduce that $P_{1, 1}$ or $Q_{1, 1}$ is identically zero, contrary to our assumption. Thus
\[
 f_{1}(x)= P_{1}(x)m_{1}(x)
 \qquad
 \text{and}
 \qquad
 g_{1}(x)= Q_{1}(x)m_{1}(x)
 \qquad
 \left(x\in \mathbb{F}\right),
\]
with appropriate generalized polynomials $P_{1}, Q_{1}\colon \mathbb{F}^{\times}\to \mathbb{C}$ and exponential $m_{1}\colon \mathbb{F}^{\times}\to \mathbb{C}$.

Suppose now that there is a positive integer $k$, less than $n$ such that for all $i=1, \ldots, k$ we have
\[
 f_{i}(x)= P_{i}(x)m_{i}(x)
 \;
 \text{and}
 \;
 g_{i}(x)= Q_{i}(x)m_{i}(x)
 \qquad
 \left(x\in \mathbb{F}^{\times}\right).
\]
Assume that in the representation of $f_{k+1}$ and $g_{k+1}$ there are different exponentials with nonzero polynomial coefficients, say $m_{j_{1}}$ and $m_{j_{2}}$. Observe that while expanding equation \eqref{Eq_inner}, the term $m_{j_{1}}^{p_{k+1}}m_{j_{2}}^{q_{k+1}}$ appears only at once, namely in the product $f_{k+1}(x^{p_{k+1}})g_{k+1}(x^{q_{k+1}})$. Again, due to Theorem \ref{T_Poly_Ind},
we deduce that the appropriate polynomial term, that is,
\[
P_{k+1, j_{1}}(x^{p_{k+1}})Q_{k+1, j_{2}}(x^{q_{k+1}})
\]
has to vanish. This proves that necessarily
\[
 f_{k+1}(x)= P_{k+1}(x)m_{k+1}(x)
 \;
 \text{and}
 \;
 g_{k+1}= Q_{k+1}(x)m_{k+1}(x)
 \qquad
 \left(x\in \mathbb{F}^{\times}\right)
\]
hold.
This shows that there exist exponentials $m_{1}, \ldots, m_{n}\colon \mathbb{F}^{\times}\to \mathbb{C}$ and generalized polynomials $P_{1}, \ldots, P_{n}$ and $Q_{1}, \ldots, Q_{n}$ on the group $\mathbb{F}^{\times}$ such that for all $i=1, \ldots, n$
\[
 f_{i}(x)= P_{i}(x)m_{i}(x)
 \qquad
 \text{and}
 \qquad
 g_{i}= Q_{i}(x)m_{i}(x)
 \qquad
 \left(x\in \mathbb{F}^{\times}\right).
\]
\end{proof}

\begin{remark}\label{r_red}
Due to conditions C(i)--C(iii), equation \eqref{Eq_inner} has the form
\begin{multline*}
 0
 =
 \sum_{i=1}^{n}f_{i}(x^{p_{i}})g_{i}(x^{q_{i}})
 \\
 =
 \sum_{i=1}^{n}P_{i}(x^{p_{i}})m_{i}(x)^{p_{i}}Q_{i}(x^{q_{i}})m_{i}(x)^{q_{i}}
 =
 \sum_{i=1}^{n}P_{i}(x^{p_{i}})Q_{i}(x^{q_{i}})m_{i}(x)^{N}
 \\
 \left(x\in \mathbb{F}^{\times}\right).
\end{multline*}
If the exponentials appearing on the right hand side would be different, then by Theorem \ref{T_Poly_Ind}, their coefficients would be zero. This implies however that there exists a proper subset
$J\subset \left\{ 1, \ldots, n\right\}$ such that
\begin{equation}\label{reduc}
 \sum_{j\in J}f_{j}(x^{p_{j}})g_{j}(x^{q_{j}})=0
 \;
 \text{as well as}
 \;
 \sum_{j\notin J}f_{j}(x^{p_{j}})g_{j}(x^{q_{j}})=0
 \qquad
 \left(x\in \mathbb{F}^{\times}\right).
\end{equation}
\end{remark}

This leads to the following definition of irreducible solutions.
\begin{definition}
A system of solutions $\left\{ f_{1}, \ldots, f_{n}, g_{1}, \ldots, g_{n}\right\}$ of equation \eqref{Eq_inner} is called \emph{irreducible} if it does not satisfy a sub-term of
\eqref{Eq_inner}. Otherwise, we say that a system of solutions is \emph{reducible}.
\end{definition}

Clearly, a system of solutions $\left\{ f_{1}, \ldots, f_{n}, g_{1}, \ldots, g_{n}\right\}$  of \eqref{Eq_inner} which fulfills \eqref{reduc} is a reducible solution. On the other hand, the argument in Remark \ref{r_red} shows that every solution of \eqref{Eq_inner} can be given as a sum of irreducible solutions of disjoint sub-terms of \eqref{Eq_inner}.
Therefore, we restrict ourselves to the irreducible case, since every reducible solution can be deduced as a sum of irreducible solutions.

\begin{cor} \label{cor_inner}
Under the conditions of Theorem \ref{theorem_inner}, suppose that system of functions
$\left\{ f_{1}, \ldots, f_{n}, g_{1}, \ldots, g_{n}\right\}$  is an irreducible solution of equation \eqref{Eq_inner}.
Then there exists an exponential $m\colon \mathbb{F}^{\times}\to \mathbb{C}$ and there are generalized polynomials $P_{i}, Q_{i}\colon \mathbb{F}^{\times}\to \mathbb{C}$ such that for each $i=1, \ldots, n$
 \begin{equation}\label{eq_irred1}
  f_{i}(x)= P_{i}(x)m(x)
  \qquad
  \text{and}
  \qquad
  g_{i}(x)= Q_{i}(x)m(x)
  \qquad
  \left(x\in \mathbb{F}^{\times}\right).
\end{equation}

In other words, for each $i=1, \ldots, n$ there exists higher order derivations $D_{i}, \widetilde{D_{i}}\colon \mathbb{F}\to \mathbb{C}$ such that
\[
 f_{i}(x)\sim D_{i}(x)
 \qquad
 \text{and}
 \qquad
 g_{i}(x)\sim \widetilde{D_{i}}(x)
 \qquad
 \left(x\in \mathbb{F}^{\times}\right),
\]
where $\sim$ in the latter two equations is the equivalence relation  defined in Remark \ref{rem2.4}.
\end{cor}

\begin{proof}
By Remark \ref{r_red}, all of the exponentials $m_i$ have to be the same in the description of the solutions of Theorem \ref{theorem_inner}. Therefore, equation \eqref{eq_irred1} describes the irreducible solutions of \eqref{Eq_inner}.

Using Lemma \ref{lem_equ_inner}, solutions of equation \eqref{Eq_inner} are enough to be determined up to the equivalence relation $\sim$ defined in Remark \ref{rem2.4}. Accordingly, we can suppose that the exponential $m$ in the above representation is the identity mapping. Hence, in view of Theorem \ref{theorem_KisLac} we get that the functions $f_{1}, \ldots, f_{n}$ as well as $g_{1}, \ldots, g_{n}$ are (or more precisely, are equivalent to) higher order derivations, as we stated.
\end{proof}

\subsection{The order of higher order derivation solutions}

Every higher order derivation on $\mathbb{F}$ is a differential operator on any finitely generated subfield of $\mathbb{F}$ (see Theorem \ref{theorem_KisLac} and \cite{KisLac18}). Hence on these fields the solutions are differential operators. Moreover, if the solutions on any finitely generated subfield of $\mathbb{F}$  are differential operators of order at most $n$, then every solution on  $\mathbb{F}$ is a derivation of order $n$. From now on, instead of finding solutions as higher order derivations we are looking for differential operators as solutions.

For this purpose, our next aim is to understand the arithmetic of the composition of derivations of the form $d_1\circ \dots \circ d_r$ that are building blocks of differential operators. First we show that there is a natural form of composition of derivations that can be taken as a standard basis.

For this target, the notion of moment function sequences turn out to be useful. Here we follow \cite{FecGseSze21}.
\noindent
A \emph{composition} of a nonnegative integer $n$ is a sequence of nonnegative integers $\alpha= \left(\alpha_{k}\right)_{k\in \mathbb{N}}$ such that
\[
 n= \sum_{k=1}^{\infty}\alpha_{k}.
\]
For a positive integer $r$, an \emph{$r$-composition} of a nonnegative integer $n$ is a composition
$\alpha= \left(\alpha_{k}\right)_{k\in \mathbb{N}}$ with
$\alpha_{k}=0$ for $k>r$.

Given a sequence of variables $x=(x_{k})_{k\in \mathbb{N}}$ and compositions $\alpha= \left(\alpha_{k}\right)_{k\in \mathbb{N}}$ and
$\beta= \left(\beta_{k}\right)_{k\in \mathbb{N}}$ we define
\[
 \alpha!=\prod_{k=1}^{\infty}\alpha_{k},\quad \lvert \alpha\rvert = \sum_{k=1}^{\infty}\alpha_{k}, \quad
 x^{\alpha}=\prod_{k=1}^{\infty}x_{k}^{\alpha_{k}},\quad \binom{\alpha}{\beta}= \prod_{k=1}^{\infty}\binom{\alpha_{k}}{\beta_{k}}.
\]
Furthermore,
$\beta \leq \alpha$ means that $\beta_{k}\leq \alpha_{k}$ for all $k\in \mathbb{N}$ and
$\beta < \alpha$ stands for $\beta \leq \alpha$  and $\beta \neq \alpha$.

\begin{definition}
Let $G$ be a commutative group, $r$ a positive integer, and for each multi-index $\alpha$ in $\mathbb{N}^r$
let $f_{\alpha}\colon G\to \mathbb{C}$ be a continuous function. We say that $(f_{\alpha})_{\alpha \in \mathbb{N}^{r}}$ is a \emph{generalized moment sequence of rank $r$}, if
\begin{equation}\label{moment}
f_{\alpha}(x+y)=\sum_{\beta\leq \alpha} \binom{\alpha}{\beta} f_{\beta}(x)f_{\alpha-\beta}(y)
\end{equation}
holds whenever $x,y$ are in $G$. The function $f_0$, where $0$ is the zero element in $\mathbb{N}^r$, is called the \emph{generating function} of the sequence.
\end{definition}

\begin{theorem}
  Let $G$ be a commutative group, $r$ a positive integer, and for each $\alpha$ in $\mathbb{N}^{r}$
 let $f_{\alpha}\colon G\to \mathbb{C}$ be a function. If the sequence of functions
 $(f_{\alpha})_{\alpha\in \mathbb{N}^{r}}$ forms a generalized moment sequence of rank $r$, then there exists an
 exponential $m\colon G\to \mathbb{C}$ and a sequence of complex-valued additive functions $a= (a_{\alpha})_{\alpha\in \mathbb{N}^{r}}$
 such that for every multi-index $\alpha$ in $\mathbb{N}^{r}$ and $x$ in $G$ we have
 \[
  f_{\alpha}(x)=B_{\alpha}(a(x))m(x),
 \]
 where $B_{\alpha}$ denotes the \emph{multivariate Bell polynomial} corresponding to the multi-index $\alpha$.
\end{theorem}

\begin{remark}\label{rembas}
 It is well-known that every polynomial $P\colon \mathbb{C}^r\to \mathbb{C}$ ($r\in \mathbb{N}$) can be given as a linear combination of Bell polynomials $B_{\alpha}$, where $\alpha\in \mathbb{N}^{r}$.
 Hence, momentum generating functions generate the exponential polynomial  functions on $G$. By Lemma \ref{L_lin_dep}, polynomials of the form $B_{\alpha}\circ a$ are linearly independent over $\mathbb{C}$.
\end{remark}

\begin{lemma}\label{L:dermom}
 Let $\mathbb{F}\subset \mathbb{C}$ be a field, $r$ be a positive integer and
 $d_{1}, \ldots, d_{r}\colon \mathbb{F}\to \mathbb{F}$ be linearly independent derivations. For all multi-index $\alpha\in \mathbb{N}^{r}$, $\alpha= \left(\alpha_{1}, \ldots, \alpha_{r}\right)$ define the function $\varphi_{\alpha}\colon \mathbb{F}\to \mathbb{C}$ by
 \begin{multline*}
  \varphi_{\alpha}(x)= d^{\alpha}(x)=
  d_{1}^{\alpha_{1}} \circ \cdots \circ d_{r}^{\alpha_{r}}(x)
  \\
  =
\underbrace{d_{1}\circ \cdots \circ d_{1}}_\text{$\alpha_{1}$ times}
\circ \cdots \circ \underbrace{d_{r}\circ \cdots \circ d_{r}}_{\text{$\alpha_{r}$ times}}(x) \\
\left(x\in \mathbb{F}^{\times}\right).
\end{multline*}
Then $(\varphi_{\alpha})_{\alpha \in \mathbb{N}^{r}}$ is a generalized moment sequence of rank $r$ on the commutative group $\mathbb{F}^{\times}$ and the generating function of the sequence is the identity function.
\end{lemma}

\begin{proof}
 Let $\mathbb{F}\subset \mathbb{C}$ be a field, $r$ be a positive integer and
 $d_{1}, \ldots, d_{r}\colon \mathbb{F}\to \mathbb{F}$ be linearly independent derivations. For all multi-index $\alpha\in \mathbb{N}^{r}$, $\alpha= \left(\alpha_{1}, \ldots, \alpha_{r}\right)$ define the function $\varphi_{\alpha}\colon \mathbb{F}\to \mathbb{C}$ as in the lemma.
 We prove the statement by induction of the length of the multi-index $\alpha$.
 Assume that $\alpha\in \mathbb{N}^{r}$ and $\lvert \alpha\rvert=1$. Then there exists $i\in \left\{1, \ldots r \right\}$ such that $\alpha_{i}=1$ and $\alpha_{j}=0$ for $j\neq i$ and
 \begin{multline*}
  \varphi_{\alpha}(xy)=
  d_{i}^{\alpha_{i}}(xy)=
  d_{i}(xy)= xd_{i}(y)+yd_{i}(x)
  \\
  =
  \varphi_{0}(x)\varphi_{\alpha}(y)+\varphi_{0}(y)\varphi_{\alpha}(x)=
  \sum_{\beta \leq \alpha}\binom{\alpha}{\beta} \varphi_{\beta}(x)\varphi_{\alpha-\beta}(y)
  \\
  \left(x, y\in \mathbb{F}^{\times}\right).
 \end{multline*}
Let now $\alpha\in \mathbb{N}^{r}$ and assume that the statement is true for all multi-indices $\beta\in \mathbb{N}^{r}$ with $\beta < \alpha$.
Then
\begin{multline*}
 \varphi_{\alpha}(xy)=  d^{\alpha}(xy)=
  d_{1}^{\alpha_{1}} \circ \cdots \circ d_{r}^{\alpha_{r}}(xy)=
\underbrace{d_{1}\circ \cdots d_{1}}_\text{$\alpha_{1}$ times}
\circ \cdots \circ \underbrace{d_{r}\circ \cdots \circ d_{r}}_{\text{$\alpha_{r}$ times}}(xy)
\\
=
\underbrace{d_{1}\circ \cdots \circ d_{1}}_\text{$\alpha_{1}$ times}
\circ \cdots \circ \underbrace{d_{r-1}\circ \cdots \circ d_{r-1}}_\text{$\alpha_{r-1}$ times}
\left(d_{r}^{\alpha_{r}}(xy)\right)
\\
=
\underbrace{d_{1}\circ \cdots \circ d_{1}}_\text{$\alpha_{1}$ times}
\circ \cdots \circ \underbrace{d_{r-1}\circ \cdots \circ d_{r-1}}_\text{$\alpha_{r-1}$ times}
\left(\sum_{\beta_r=0}^{\alpha_{r}}\binom{\alpha_{r}}{\beta_r}d_{r}^{\beta_r}(x)\cdot d^{\alpha_{r}-\beta_r}_{r}(y)\right)
\\=
\sum_{\beta_r=0}^{\alpha_{r}}\binom{\alpha_{r}}{\beta_r}\underbrace{d_{1}\circ \cdots \circ d_{1}}_\text{$\alpha_{1}$ times}
\circ \cdots \circ \underbrace{d_{r-1}\circ \cdots \circ d_{r-1}}_\text{$\alpha_{r-1}$ times}
\left(d_{r}^{\beta_r}(x)\cdot d^{\alpha_{r}-\beta_r}_{r}(y)\right)
\\=
\sum_{\beta_1=0}^{\alpha_{1}}\dots\sum_{\beta_{r-1}=0}^{\alpha_{r-1}}\sum_{\beta_r=0}^{\alpha_{r}}\binom{\alpha_{1}}{\beta_1}\dots \binom{\alpha_{r-1}}{\beta_{r-1}}\binom{\alpha_{r}}{\beta_r} \times
\\
\times d_{1}^{\beta_{1}} \circ \cdots \circ d_{r}^{\beta_{r}}(x)~\cdot~ d_{1}^{\alpha_{1}-\beta_1} \circ \cdots \circ d_{r}^{\alpha_{r}-\beta_r}(y)
\\
=
\sum_{\beta \leq \alpha}\binom{\alpha}{\beta}\varphi_{\beta}(x)\varphi_{\alpha-\beta}(y)
\qquad
\left(x, y\in \mathbb{F}^{\times}\right).
\end{multline*}
\end{proof}

\begin{cor} \label{cor:alg_ind}
By Lemma \ref{L:dermom} and Remark \ref{rembas} implies that the functions $d^{\alpha}=
  d_{1}^{\alpha_{1}} \circ \cdots \circ d_{r}^{\alpha_{r}}$ constitute a basis of the differential operators in $\mathbb{F}^{\times}$. Since every $d^{\alpha}$ is additive, by Lemma \ref{L_lin_dep},
  all elements of the system $\{d^{\alpha}(x)\}$, where $\alpha\in \cup_{r\in \mathbb{N}}\mathbb{N}^r$ are algebraically independent.

   A consequence of the algebraic independence of the elements of $d^{\alpha}$, where $\alpha\in \cup_{r\in \mathbb{N}}\mathbb{N}^r$ is the following. Let $P\in \mathbb{C}[x_1, \dots, x_n]$ be a polynomial, $d_1, \dots d_r$ be derivations as in Lemma \ref{L:dermom} and $\alpha_1, \dots, \alpha_n \in  \cup_{r\in \mathbb{N}}\mathbb{N}^r$. Then the
   following polynomial form
    \[
     P(d^{\alpha_1}(x), \ldots, d^{\alpha_n}(x))=0
    \qquad
     \left(x\in G\right)
    \]
    holds if and only if
    \[
     P(\hat{d}^{~\lvert \alpha_1\rvert}(x), \ldots, \hat{d}^{~\lvert \alpha_n\rvert}(x))=0
    \qquad
     \left(x\in G\right),
    \] where $\hat{d}$ is an arbitrary derivation (of order 1).
    In other words, we can substitute $d_1, \dots, d_r$ by $\hat{d}$ in $d^{\alpha_1}, \ldots, d^{\alpha_n}$.
\end{cor}

By Corollary \ref{cor:alg_ind}, it would be desirable to calculate $d^{k}=
  \underbrace{d\circ \cdots \circ d}_\text{$k$ times}$, where $d$ is a derivation (of order 1) and $k\in \mathbb{N}$.
Lemma \ref{L:dermom}, together with \cite[Proposition 1]{FecGseSze21}, implies the following statement.

\begin{proposition}\label{prop1}
 Let $\mathbb{F}\subset \mathbb{C}$ be a field and $d\colon \mathbb{F}\to \mathbb{C}$ a derivation. For all positive integer $k$ we define the function $d^{k}$ on
 $\mathbb{F}$ by
 \[
  d^{k}(x)=
  \underbrace{d\circ \cdots \circ d}_\text{$k$ times}(x)
  \qquad
  \left(x\in \mathbb{F}\right).
 \]
Then for all positive integer $p$ we have
\[
 d^{k}(x_{1}\cdots x_{p})=
 \sum_{\substack{l_{1}, \ldots, l_{p}\geq 0\\ l_{1}+\cdots+l_{p}=k}}
 \binom{k}{l_{1}, \ldots, l_{p}} \prod_{t=1}^{p} d^{l_{t}}(x_{t})
 \qquad
 \left(x_{1}, \ldots, x_{p}\in \mathbb{F}\right),
\]
where the conventions $d^{0}= \mathrm{id}$ and $\displaystyle\binom{k}{l_{1}, \ldots, l_{p}}= \dfrac{k!}{l_{1}! \cdots l_{p}!}$ are  adopted. Especially, for all positive integer $p$, we have
\[
d^{k}(x^{p})=
 \sum_{\substack{l_{1}, \ldots, l_{p}\geq 0\\ l_{1}+\cdots+l_{p}=k}}
 \binom{k}{l_{1}, \ldots, l_{p}}\cdot d^{l_{1}}(x)\cdots d^{l_{p}}(x)
 \qquad
 \left(x_{1}, \ldots, x_{p}\in \mathbb{F}\right),
\]

Reordering the previous expression we can get the following

 \begin{multline*}
 d^{k}(x^{p})
  = \sum_{\substack{ j_{1}+\cdots+j_{s}=p'< p\\
 j_1+2j_2+\dots+sj_s=k}}  \binom{k}{\underbrace{1,\dots, 1}_{j_{1}}, \ldots, \underbrace{s,\dots, s}_{j_{s}}}
 \prod_{t=1}^s \frac{1}{(j_t!)}\times\\
\times \binom{p}{\underbrace{1,\dots,1}_{p'}}\cdot (d(x))^{j_1}\cdots (d^{s}(x))^{j_s}\cdot x^{p-p'},\qquad
 \left(x\in \mathbb{F}\right)
 \end{multline*}
where $j_1,\dots, j_s$ denotes the number of $d(x),\dots, d^s(x)$ in a given composition of $d^k(x^p)$.
\end{proposition}

With the above results, we are now ready to prove an upper bound for the order of derivations appearing in Theorem \ref{theorem_inner}.

\begin{theorem}\label{theorem_inner_degree}
 Under the hypotheses of Theorem \ref{theorem_inner}, the solutions $f_i$ and $g_i$ are derivations $D_{i}$ and $\widetilde{D}_{i}$ for all $i=1, \ldots, n$. Let $k$ and $l$ denote the maximal orders of derivations $D_{i}$ and $\widetilde{D}_{i}$, respectively. Suppose that there exists some $i'$ such that the order of $D_{i'}$ and $\widetilde{D}_{i'}$ is exactly $k$ and $l$, respectively. Then for all $i=1, \ldots, n$ the order of $D_{i}$  and $\widetilde{D}_{i}$ is less or equal to $n-1$.
\end{theorem}

\begin{proof}
Assume contrary that the maximal order $k$ of the above derivations $D_{1}, \ldots, D_{n}$ is greater than $n-1$. The argument for the case when the maximal order $l$ of $\widetilde{D}_{1}, \ldots, \widetilde{D}_{n}$ is greater than $n-1$ is analogous.

By our assumption there exists an index $i'$ such that the orders of the derivations $D_{i'}$ and $\widetilde{D}_{i'}$ is exactly $k$ and $l$, respectively. It is important to note that then the sum of the orders of $D_{i}$ and $\widetilde{D}_{i}$ is at most $k+l$ for any $i$, and it is exactly $k+l$ for some indices if and only if the corresponding $D$ and $\widetilde{D}$ is of order $k$ and of order $l$, respectively. Furthermore, by the algebraic independence of higher order derivations, there the number of these indices are at least two.

From now on we assume that $\mathbb{F}$ is finitely generated. Indeed, if we verify the statement for any finitely generated subfield of a field $\mathbb{F}$, then it holds for $\mathbb{F}$ itself, as well.
On finitely generated fields all of these higher order derivations can be represented as differential operators, that is, on finitely generated fields we have
\[
 D_{i}(x)= \sum_{\lvert \alpha\rvert\le k}\lambda_{i,\alpha}d_i^{\alpha}(x)
 \qquad
 \text{and}
 \qquad
 \widetilde{D}_{i}(x)= \sum_{\lvert \beta\rvert\le l}\widetilde{\lambda}_{i,\beta}\widetilde {d}_i^{\beta}(x)
 \qquad
 \left(x\in \mathbb{F}\right)
\]
with appropriate complex constants $\lambda_{i,\alpha}$, $\widetilde{\lambda}_{i, \beta}$, ($\lvert \alpha\rvert\le k$, $\lvert \beta\rvert\le l$, $i=1, \ldots, n$) and higher order derivations $d_{i}^{\alpha}, \widetilde{d}_{i}^{\beta} \colon \mathbb{F}\to \mathbb{C}$ defined in Lemma \ref{L:dermom}.
Further, we have
\[
0=
 \sum_{i=1}^{n}D_{i}(x^{p_{i}})\widetilde{D}_{i}(x^{q_{i}})
 \\
 =
 \sum_{i=1}^{n}\left(\sum_{\lvert \alpha\rvert\le k}\lambda_{i, \alpha} d_i^{\alpha}(x^{p_{i}}) \right)
 \cdot \left(\sum_{\lvert \beta\rvert\le l}\widetilde{\lambda}_{i, \beta}\widetilde{d}_i^{\beta}(x^{q_{i}})\right)
 \qquad
 \left(x\in \mathbb{F}\right).
\]
If we expand the right hand side of the above identity with the aid of Proposition \ref{prop1}, we get an expression of the following polynomial form
\begin{multline*}
 P(x, d_{1}(x), \ldots, d_{k}(x), \widetilde{d}_{1}(x), \ldots, \widetilde{d}_{l}(x), \ldots, d_{1}^{k}(x), \ldots,
 \\
 d_{k}^{k}(x), \widetilde{d}^{l}_{1}(x), \ldots, \widetilde{d}^{l}_{l}(x))=0
 \qquad
 \left(x\in \mathbb{F}\right).
\end{multline*}
If this identity can be satisfied by different functions, it can also be satisfied by a single one. By Corollary \ref{cor:alg_ind}, this enables us to substitute the functions $d_{i}^{\alpha}, \widetilde{d}_{i}^{\beta} \, (i=1, \ldots, n, \lvert \alpha\rvert\le k, \lvert \beta\rvert\le l)$ with suitable compositions of a given derivation $d$ of order 1.
In other words, instead of the above identity, we can restrict ourselves to
\[
 \sum_{i=1}^{n}\left(\sum_{j=0}^{k}\lambda_{i, j}d^{j}(x^{p_{i}}) \right)
 \cdot \left(\sum_{j=0}^{l}\widetilde{\lambda}_{i, j}d^{j}(x^{q_{i}})\right)=0
 \qquad
 \left(x\in \mathbb{F}\right)
\]
with appropriate complex constants $\lambda_{i, j}$ \, ($i=1, \ldots, n$, $j=0, \ldots, k$), and  $\widetilde{\lambda}_{i, j}$\, ($i=1, \ldots, n$, $j=0, \ldots, l$) and derivation $d \colon \mathbb{F}\to \mathbb{C}$. By our assumptions there are some $i'$  such that $\la_{i',k}\ne 0$ and $\tla_{i', l}\ne 0$.

Dividing the above sum to smaller ones, we get
\begin{multline*}
 \sum_{i=1}^{n}\left(\sum_{j=0}^{k}\lambda_{i, j}d^{j}(x^{p_{i}}) \right)
 \cdot \left(\sum_{j=0}^{l}\widetilde{\lambda}_{i, j}d^{j}(x^{q_{i}})\right)
 \\
 =
 \sum_{i=1}^{n}\left(\sum_{j=0}^{k-1}\lambda_{i, j}d^{j}(x^{p_{i}})+\lambda_{i, k}d^{k}(x^{p_{i}})  \right)
 \cdot \left(\sum_{j=0}^{k-1}\widetilde{\lambda}_{i, j}d^{j}(x^{q_{i}})+\widetilde{\lambda}_{i, l}d^{l}(x^{q_{i}})\right)
  \\
  \sum_{i=1}^{n}
  \left[S({p_i}, k-1)S(q_{i}, l-1)+\widetilde{\lambda}_{i, l}d^{l}(x^{q_{i}})S(p_{i}, k-1)
  \right.
  \\
  \left.
  + \lambda_{i, k}d^{k}(x^{p_{i}})S(q_{i}, l-1)+\lambda_{i, k}\widetilde{\lambda}_{i, k}d^{k}(x^{p_{i}})d^{l}(x^{q_{i}})\right]=0
 \\
 \left(x\in \mathbb{F}\right),
\end{multline*}
where
\[
 S(p_{i}, k-1)= \sum_{j=0}^{k-1}\lambda_{i, j}d^{j}(x^{p_{i}})
 \;
 \text{and}
 \;
 S(q_{i}, l-1)= \sum_{j=0}^{l-1}\widetilde{\lambda}_{i, j}d^{j}(x^{q_{i}})
 \qquad
 \left(x\in G\right).
\]
Note that, by the algebraic independence used in Corollary \ref{cor:alg_ind}, this sum splits into separate terms of the form $d^{s}(x^{p_i})d^{t}(x^{q_i})$, where $s+t$ is a fixed number. By our assumption, when $s+t=k+l$,  then the only way is that $s= k, t= l$.
This implies that using Proposition \ref{prop1} we get that
\small{
\begin{multline*}
0=\sum_{i=1}^{n} \lambda_{i, k}\widetilde{\lambda}_{i, l}d^{k}(x^{p_{i}})d^{l}(x^{q_{i}})
 \\
 =
 \sum_{i=1}^{n}\lambda_{i, k}\widetilde{\lambda}_{i, k}
 \times
 \\
 \times
  \left(\sum_{\substack{ j_{1}+\cdots+j_{s}=p'< p_i\\
 j_1+2j_2+\dots+sj_s=k}}  \binom{k}{\underbrace{1,\dots, 1}_{j_{1}}, \ldots, \underbrace{s,\dots, s}_{j_{s}}}  \binom{p_i}{\underbrace{1,\dots,1}_{p'}}
 \prod_{t=1}^s \frac{(d^t(x))^{j_t}}{(j_t!)} x^{p_i-p'}\right)\times
 \\
 \times
 \left(\sum_{\substack{ j_{1}+\cdots+j_{s}=q'< q_i\\
 j_1+2j_2+\dots+sj_s=l}}  \binom{l}{\underbrace{1,\dots, 1}_{j_{1}}, \ldots, \underbrace{s,\dots, s}_{j_{s}}} \binom{q_i}{\underbrace{1,\dots,1}_{q'}}
 \prod_{t=1}^s \frac{(d^t(x))^{j_t}}{(j_t!)} x^{q_i-q'}\right)
\end{multline*}
}
\normalsize
while the rest in the above sum can be computed similarly.

\medskip
\noindent
Case 1. If
$k<l$, then we compute the coefficients of the terms
\begin{equation}\label{term1}(d(x))^{j} d^{k-j}(x)d^{l}(x) ~ ~ ~ (j=0,\dots, k-1).
\end{equation}
For each $j=0, \dots, k-1$ this can be taken from the expansion of  \[\lambda_{i, k}\widetilde{\lambda}_{i, l}d^{k}(x^{p_{i}})d^{l}(x^{q_{i}})\] in only one way. Namely, by splitting $d^{k}(x^{p_{i}})$ into $j+1$ parts and $d^{l}(x^{q_{i}})$ into one. Then the corresponding coefficients are
$$\sum_{i=1}^n\lambda_{i, k}\widetilde{\lambda}_{i, l}\binom{k}{\underbrace{1,\dots, 1}_{j}}\frac{1}{j!}\binom{p_i}{\underbrace{1,\dots,1}_{j+1}} \binom{q_i}{1}=0.$$
Since each of the terms contains $\displaystyle \binom{k}{\underbrace{1,\dots, 1}_{j}}\dfrac{1}{j!}=\binom{k}{j}$, this can be eliminated from the above equation.
The corresponding equations ($j=0,\dots, k$) can be written in the following matrix equation
\[
 \begin{pmatrix}
 \binom{p_{1}}{1}\binom{q_{1}}{1} & \ldots & \binom{p_{n}}{1}\binom{q_{n}}{1}\\
 \binom{p_{1}}{1, 1}\binom{q_{1}}{1} & \ldots & \binom{p_{n}}{1, 1}\binom{q_{n}}{1} \\
 \vdots & \ddots & \vdots\\
  \binom{p_{1}}{\underbrace{1, \ldots, 1}_k}\binom{q_{1}}{1} & \ldots & \binom{p_{n}}{\underbrace{1, \ldots, 1}_k}\binom{q_{n}}{1}
 \end{pmatrix}
 \cdot
 \begin{pmatrix}
  \lambda_{1, k}\widetilde{\lambda}_{1, l}\\
  \lambda_{2, k}\widetilde{\lambda}_{2, l}\\
  \vdots \\
  \lambda_{n, k}\widetilde{\lambda}_{n, l}
 \end{pmatrix}
=
\begin{pmatrix}
 0\\
 0\\
 \vdots
 \\
 0
\end{pmatrix}.
 \]
Here we note that as $p_i$'s are all different positive integers, it follows that $p_i\ge n$ for some $i$ and hence the first $n$ rows of the matrix are not identically zero, as $k>n-1$.
It is straightforward to verify that the first $n$ row of above matrix equation is equivalent to
\[
 \begin{pmatrix}
 p_{1} & \ldots & p_{n}\\
 p_{1}^2 & \ldots & p_{n}^{2} \\
 \vdots & \ddots & \vdots\\
 p_{1}^{n} & \ldots & p_{n}^{n}
 \end{pmatrix}
 \cdot
 \begin{pmatrix}
 q_1\lambda_{1, k}\widetilde{\lambda}_{1, l}\\
  q_2\lambda_{2, k}\widetilde{\lambda}_{2, l}\\
  \vdots \\
  q_n\lambda_{n, k}\widetilde{\lambda}_{n, l}
 \end{pmatrix}
=
\begin{pmatrix}
 0\\
 0\\
 \vdots
 \\
 0
\end{pmatrix}.
 \]
Since this is a Vandermonde type matrix with different $p_i$'s, the only solution of this homogeneous linear system is the zero vector, i.e., $ q_i\lambda_{i, k}\widetilde{\lambda}_{i, l}= 0$  for all $i=1, \ldots, n$. This contradicts to our assumption that there is some $i'$ for which $\lambda_{i', k}\ne 0$ and $\widetilde{\lambda}_{i', l}\ne 0$ ($q_{i'}\ne 0$ as $q_i\ne 0$ for all $i\in \{1, \dots, k\}$).

\medskip
\noindent
Case 2.
If $k=l$, then we compute the coefficients of the terms \begin{equation}\label{term2}(d(x))^{2j} (d^{k-j}(x))^2
\qquad
(j=0,\dots, k-1).
\end{equation}

If $j<\dfrac{k}{2}$, then this term can be taken from the expansion of \[\lambda_{i, k}\widetilde{\lambda}_{i, l}d^{k}(x^{p_{i}})d^{k}(x^{q_{i}})\] in only one way.
Namely, by splitting $d^{k}(x^{p_{i}})$ and $d^{k}(x^{q_{i}})$ into $j+1$ parts. Then the corresponding coefficients are $$\sum_{i=1}^n\lambda_{i, k}\widetilde{\lambda}_{i, l}\left(\binom{k}{\underbrace{1,\dots, 1}_{j}}\frac{1}{j!}\right)^2\binom{p_i}{\underbrace{1,\dots,1}_{j+1}} \binom{q_i}{\underbrace{1,\dots,1}_{j+1}}=0.$$
Since each of the terms contains $\Big(\binom{k}{\underbrace{1,\dots, 1}_{j}}\dfrac{1}{j!}\Big)^2=\binom{k}{j}^2$, this can be eliminated from the above equations.
These equations for $j=0,\dots, \allowbreak \lceil k/2\rceil -1$ can be written in the following matrix equation
\small{
\[
 \begin{pmatrix}
 \binom{p_{1}}{1}\binom{q_{1}}{1} & \ldots & \binom{p_{n}}{1}\binom{q_{n}}{1}\\
 \binom{p_{1}}{1,1}\binom{q_{1}}{1,1} & \ldots & \binom{p_{n}}{1,1}\binom{q_{n}}{1, 1}\\
 \vdots & \ddots & \vdots\\
  \binom{p_{1}}{\underbrace{1, \ldots, 1}_{\lceil k/2\rceil -1}}\binom{q_{1}}{\underbrace{1, \ldots, 1}_{\lceil k/2\rceil -1}} & \ldots & \binom{p_{n}}{\underbrace{1, \ldots, 1}_{\lceil k/2\rceil -1}}\binom{q_{n}}{\underbrace{1, \ldots, 1}_{\lceil k/2\rceil -1}}
 \end{pmatrix}
 \cdot
 \begin{pmatrix}
  \lambda_{1, k}\widetilde{\lambda}_{1, k}\\
  \lambda_{2, k}\widetilde{\lambda}_{2, k}\\
  \vdots \\
  \lambda_{n, k}\widetilde{\lambda}_{n, k}
 \end{pmatrix}
=
\begin{pmatrix}
 0\\
 0\\
 \vdots
 \\
 0
\end{pmatrix}.
 \]
 } \normalsize
Here we note that as $p_i, q_i$ are all different positive integers, it follow that  $\max_i \left\{ p_i,q_i\right\} \ge 2n$ and hence the first $n'=\min(n, \lceil k/2\rceil -1)$ rows of the matrix in not identically zero, as $k>n-1$.
It is straightforward to verify that the first $n'$ row of above matrix equation is equivalent to
\[
 \begin{pmatrix}
p_1q_1& \ldots &p_nq_n\\
 p_{1}^2q_1^2 & \ldots & p_{n}^{2}q_n^2 \\
 \vdots & \ddots & \vdots\\
 p_{1}^{n'}q_1^{n'} & \ldots & p_{n}^{n'}q_n^{n'}
 \end{pmatrix}
 \cdot
 \begin{pmatrix}
 \lambda_{1, k}\widetilde{\lambda}_{1, k}\\
 \lambda_{2, k}\widetilde{\lambda}_{2, l}\\
  \vdots \\
 \lambda_{n, k}\widetilde{\lambda}_{n, k}
 \end{pmatrix}
=
\begin{pmatrix}
 0\\
 0\\
 \vdots
 \\
 0
\end{pmatrix}.
 \]
Note that if $n'=n$, then we are done with a Vandermonde matrix argument similar as it is used in Case 1. So from now on we assume that $n'=\lceil k/2\rceil -1$. This also means that $k<2n$, thus the maximal order of the corresponding derivations is at most $2n-1$.

\medskip
If $j\ge \dfrac{k}{2}$, then the term $(d(x))^{2j} (d^{k-j}(x))^2$ can be taken from the expansion of \[\lambda_{i, k}\widetilde{\lambda}_{i, l}d^{k}(x^{p_{i}})d^{k}(x^{q_{i}})\] in three ways.
One is as above, when we split both $d^{k}(x^{p_{i}})$ and $d^{k}(x^{q_{i}})$ into $j+1$ parts.
Another one is when $d^{k}(x^{p_{i}})$ is split into $k$ parts giving $d(x)^k$, and $d^{k}(x^{q_{i}})$ provides $(d^{k-j}(x))^2(d(x))^{2j-k}$. The third one is given by changing the role of $p_i$ and $q_i$.

Then the corresponding coefficients are
\begin{align*}
&\sum_{i=1}^n\lambda_{i, k}\widetilde{\lambda}_{i, k}\Big(\binom{k}{\underbrace{1,\dots, 1}_{j}}\frac{1}{(j!)}\Big)^2\binom{p_i}{\underbrace{1,\dots,1}_{j}} \binom{q_i}{\underbrace{1,\dots,1}_{j}}+\\
&\sum_{i=1}^n\lambda_{i, k}\widetilde{\lambda}_{i, k}
\binom{k}{\underbrace{1,\dots, 1}_{k}}\frac{1}{(k!)}\binom{k}{\underbrace{1,\dots, 1}_{2j-k}, j}\frac{1}{2}\frac{1}{(2j-k)!}\times \\
& \times \Big(\binom{p_i}{\underbrace{1,\dots,1}_{k}} \binom{q_i}{\underbrace{1,\dots,1}_{2j-k+2}}+ \binom{p_i}{\underbrace{1,\dots,1}_{2j-k+2}} \binom{q_i}{\underbrace{1,\dots,1}_{k}}\Big)=0.
\end{align*}

\noindent
First we show that the second sum has to vanish for all $j=\lceil k/2\rceil, \dots, k-1$. In such cases, the coefficients \[\binom{k}{\underbrace{1,\dots, 1}_{k}}\frac{1}{k!}\binom{k}{\underbrace{1,\dots, 1}_{2j-k}, j}\frac{1}{2}\frac{1}{(2j-k)!}\] are the same is each summand, it is enough to show that
\begin{equation}\label{eq_m1}\sum_{i=1}^n\lambda_{i, k}\widetilde{\lambda}_{i, k}
\Big(\binom{p_i}{\underbrace{1,\dots,1}_{k}} \binom{q_i}{\underbrace{1,\dots,1}_{2j-k+2}}+ \binom{p_i}{\underbrace{1,\dots,1}_{2j-k+2}} \binom{q_i}{\underbrace{1,\dots,1}_{k}}\Big)=0.\end{equation}
This clearly holds, since the term $(d(x))^{2j+1}d^{2k-2j-1}(x)$ in the expansion of $d^k(x^{p_i}) d^k(x^{q_i})$ for $j=\lceil k/2\rceil, \dots, k-1$ can be given in exactly two ways. Either $(d(x))^k$ stems from $d^k(x^{p_i})$ and  $(d(x))^{2j-k+1}\cdot \allowbreak d^{2k-2j-1}(x)$ stems from $d^k(x^{q_i})$, or reversely changing the role of $p_i$ and $q_i$.  Hence we get
\begin{multline*}
\sum_{i=1}^n\lambda_{i, k}\widetilde{\lambda}_{i, k}\binom{k}{\underbrace{1,\dots,1}_{2j-k+1}}\cdot \frac{1}{(2j-k+1)!}
\times\\
\times
\Big(\binom{p_i}{\underbrace{1,\dots,1}_{k}} \binom{q_i}{\underbrace{1,\dots,1}_{2j-k+2}}+ \binom{p_i}{\underbrace{1,\dots,1}_{2j-k+2}} \binom{q_i}{\underbrace{1,\dots,1}_{k}}\Big)=0,
\end{multline*}
which is equivalent to equation \eqref{eq_m1}. Thus, for all $j=\lceil k/2\rceil, \dots, k-1$ it follows
$$ \sum_{i=1}^n\lambda_{i, k}\widetilde{\lambda}_{i, k}\Big(\binom{k}{\underbrace{1,\dots, 1}_{j}}\frac{1}{(j!)}\Big)^2\binom{p_i}{\underbrace{1,\dots,1}_{j}} \binom{q_i}{\underbrace{1,\dots,1}_{j}}=0$$

This implies that $$ \sum_{i=1}^n\lambda_{i, k}\widetilde{\lambda}_{i, k}\binom{p_i}{\underbrace{1,\dots,1}_{j}} \binom{q_i}{\underbrace{1,\dots,1}_{j}}=0$$
hold for all $j=0,\dots, k-1$. Thus we get
\small{
\[
 \begin{pmatrix}
 \binom{p_{1}}{1}\binom{q_{1}}{1} & \ldots & \binom{p_{n}}{1}\binom{q_{n}}{1}\\
 \binom{p_{1}}{1,1}\binom{q_{1}}{1,1} & \ldots & \binom{p_{n}}{1,1}\binom{q_{n}}{1, 1}\\
 \vdots & \ddots & \vdots\\
  \binom{p_{1}}{\underbrace{1, \ldots, 1}_{k}}\binom{q_{1}}{\underbrace{1, \ldots, 1}_{k}} & \ldots & \binom{p_{n}}{\underbrace{1, \ldots, 1}_{k}}\binom{q_{n}}{\underbrace{1, \ldots, 1}_{k}}
 \end{pmatrix}
 \cdot
 \begin{pmatrix}
  \lambda_{1, k}\widetilde{\lambda}_{1, k}\\
  \lambda_{2, k}\widetilde{\lambda}_{2, k}\\
  \vdots \\
  \lambda_{n, k}\widetilde{\lambda}_{n, k}
 \end{pmatrix}
=
\begin{pmatrix}
 0\\
 0\\
 \vdots
 \\
 0
\end{pmatrix},
 \]} \normalsize
where the first $n$ rows of the matrix is not identically zero as $p_i$ and $q_i$ are different, hence there is an index $i'$ such that $p_{i'}\ge n$ and $q_{i'}\ge n$.
Thus this system consisting the first $n$ rows is equivalent to

\[
 \begin{pmatrix}
p_1q_1& \ldots &p_nq_n\\
 p_{1}^2q_1^2 & \ldots & p_{n}^{2}q_n^2 \\
 \vdots & \ddots & \vdots\\
 p_{1}^{n}q_1^{n} & \ldots & p_{n}^{n}q_n^{n}
 \end{pmatrix}
 \cdot
 \begin{pmatrix}
 \lambda_{1, k}\widetilde{\lambda}_{1, k}\\
 \lambda_{2, k}\widetilde{\lambda}_{2, l}\\
  \vdots \\
 \lambda_{n, k}\widetilde{\lambda}_{n, k}
 \end{pmatrix}
=
\begin{pmatrix}
 0\\
 0\\
 \vdots
 \\
 0
\end{pmatrix}.
 \]
By the usual Vandermonde argument as in Case 1, the only solution of this homogeneous linear system is the zero vector, i.e., $ \lambda_{i, k}\widetilde{\lambda}_{i, l}= 0$  for all $i=1, \ldots, n$. This contradicts our assumption that there is some $i'$ for which $\lambda_{i', k}\ne 0$ and $\widetilde{\lambda}_{i', k}\ne 0$.

\medskip
\noindent
Case 3. If $k>l$, then we prove by induction in $j$ that $$\sum_{i=1}^n \la_{i,k}\tla_{i,l}p_i^{j+1-s}q_i^{s}=0$$
holds for every $j=0,\dots, n-1$ and every $s=0,\dots, j$.

In each step we consider how $(d(x))^{j}d^{k-j}(x)d^{l}(x)$ can be given from the expansion $d^{k}(x^{p_i})d^{l}(x^{q_i})$. There are three possible  ways, where this term can stem from.

\begin{enumerate}[(a)]
    \item $(d(x))^{j}d^{k-j}(x)$ stems from $d^{k}(x^{p_i})$ and $d^{l}(x)$ stems from $d^{l}(x^{q_i})$. This can happen for every $j\in \{0, \dots, k\}$.
    In this case the coefficient of $(d(x))^{j}d^{k-j}(x)d^{l}(x)$ is
    \[
    \sum_{i=1}^n \la_{i,k}\tla_{i,l}\binom{k}{\underbrace{1,\dots,1}_{j}}\frac{1}{j!}\binom{p_i}{\underbrace{1,\dots,1}_{j+1}}\binom{q_i}{1}
    \\
    =\sum_{i=1}^n \la_{i,k}\tla_{i,l}\binom{k}{j}\binom{p_i}{\underbrace{1,\dots,1}_{j+1}}q_i.
    \]
    \item $d^{k-j}(x)d^{l}(x)(d(x))^{j-l}$ stems from  $d^{k}(x^{p_i})$ and $(d(x))^{l}$ stems from $d^{l}(x^{q_i})$. This can happen if $j\ge l$.
    In this case the coefficient of $(d(x))^{j}d^{k-j}(x)d^{l}(x)$ is
    \begin{align*}&\sum_{i=1}^n \la_{i,k}\tla_{i,l}\binom{k}{k-j,l,\underbrace{1,\dots,1}_{j-l}}\frac{1}{(j-l)!}\binom{l}{\underbrace{1,\dots,1}_{l}}\frac{1}{l!} \binom{p_i}{\underbrace{1,\dots,1}_{j-l+2}}\binom{q_i}{\underbrace{1,\dots,1}_{l}} \\
    &=\sum_{i=1}^n \la_{i,k}\tla_{i,l}\binom{k}{k-j,j-l,l}\binom{p_i}{\underbrace{1,\dots,1}_{j-l+2}}\binom{q_i}{\underbrace{1,\dots,1}_{l}}.\end{align*}
    \item $d^{l}(x)(d(x))^{k-l}$ stems from  $d^{k}(x^{p_i})$ and $(d(x))^{l-(k-j)}d^{k-j}(x)$ stems from $d^{l}(x^{q_i})$. This can happen if $l\ge k-j$.
    In this case the coefficient of $(d(x))^{j}d^{k-j}(x)d^{l}(x)$ is
    \begin{multline*}\sum_{i=1}^n \la_{i,k}\tla_{i,l}\binom{k}{ k-l,\underbrace{1,\dots,1}_{l}}\frac{1}{l!}\binom{l}{k-j,\underbrace{1,\dots,1}_{l-(k-j)}}\frac{1}{(l-(k-j))!}\times\\
    \times
    \binom{p_i}{\underbrace{1,\dots,1}_{l+1}}\binom{q_i}{\underbrace{1,\dots,1}_{l-(k-j)+1}}=\\
    =\sum_{i=1}^n \la_{i,k}\tla_{i,l}\binom{k}{k-l,k-j,l-(k-j)}\binom{p_i}{\underbrace{1,\dots,1}_{l+1}}\binom{q_i}{\underbrace{1,\dots,1}_{l-(k-j)+1}}.\end{multline*}
For $j=0$ (and hence $s=0$) only the first term takes into account. This means that
$$\sum_{i=1}^n \la_{i,k}\tla_{i,l}p_iq_i=0.$$ So the inductive hypothesis  holds for $j=0$.

Now we assume that $$\sum_{i=1}^n \la_{i,k}\tla_{i,l}p_i^{j'+1-s}q_i^{s}=0$$
holds for every $j'=0,\dots, j-1$ and every $s=0,\dots, j'$. We prove that  $$\sum_{i=1}^n \la_{i,k}\tla_{i,l}p_i^{j+1-s}q_i^{s}=0$$
holds for every $s=0,\dots, j$, as well.
Generally, some of the previous compositions are possible for a given $j$ but the following argument works in all cases, however we just prove it when all compositions discussed below appear in the expansion.
Thus we assume that the coefficient of $(d(x))^{j}d^{k-j}(x)d^{l}(x)$ is
\begin{multline*}
0=\sum_{i=1}^n \la_{i,k}\tla_{i,l}\left[\binom{k}{j}\binom{p_i}{\underbrace{1,\dots,1}_{j+1}}q_i \right.\\
\left. + \binom{k}{k-j,j-l,l}\binom{p_i}{\underbrace{1,\dots,1}_{j-l+2}}\binom{q_i}{\underbrace{1,\dots,1}_{l}} \right.
\\
\left.
+\binom{k}{k-l,k-j,l-(k-j)}\binom{p_i}{\underbrace{1,\dots,1}_{l+1}}\binom{q_i}{\underbrace{1,\dots,1}_{l-(k-j)+1}}.\right]\end{multline*}

By the inductive hypothesis and the fact that $j<n\le\min \left\{ p_i',q_i'\right\}$ for some $i'$, this is equivalent to
\begin{multline*}
0=
\sum_{i=1}^n \la_{i,k}\tla_{i,l}\left(\binom{k}{j}p_i^{j+1}q_i+ \binom{k}{k-j,j-l,l}p_i^{j-l+2}q_i^{l} \right.
\\
\left.
+\binom{k}{k-l,k-j,l-(k-j)}p_i^{l+1}q_i^{l-(k-j)+1}\right).
\end{multline*}
Note that the expressions $p_i^{j+1-s}q_i^{s}$ for $s=0, \dots, j$ can be interchanged in the following sense. By C(ii), $p_i+q_i=N$, thus we have that $p_i^{j+1-s}q_i^{s}=Np_i^{j-s}q_i^{s}-p_i^{j-s}q_i^{s+1}$. As $$\sum_{i=1}^n\la_{i,k}\tla_{i,l}p_i^{j-s}q_i^{s}=0$$ for every $s=0,\dots, j-1$ by the inductive hypothesis, the term $Np_i^{j-s}q_i^{s}$ can be eliminated. After several repetition of this step we get that
\begin{multline*}
0=\sum_{i=1}^n \la_{i,k}\tla_{i,l}\left(\binom{k}{j}p_i^{j+1}q_i\pm \binom{k}{k-j,j-l,l}p_i^{j+1}q_i
\right.
\\
\left.
\pm \binom{k}{k-l,k-j,l-(k-j)}p_i^{j+1}q_i\right).
\end{multline*}
We also note that using this interchange rule and the inductive hypothesis it is clear that $$\sum_{i=1}^n \la_{i,k}\tla_{i,l}p_i^{j+1-s}q_i^{s}=0$$
for any $s=0,\dots, j$ is equivalent to $$\sum_{i=1}^n \la_{i,k}\tla_{i,l}p_i^{j+1}q_i=0.$$
Therefore, to finish the proof it is enough to show that $$\binom{k}{j}\pm \binom{k}{k-j,j-l,l}\pm \binom{k}{k-l,k-j,l-(k-j)}\ne 0.$$
This is equivalent to verify that
$$\frac{1}{j!}\pm \frac{1}{(j-l)!(l!)} \pm\frac{1}{(k-l)!(l-(k-j)!)}\ne 0.$$
Multiplying by $j!$ this lead to
$$1\pm \binom{j}{l}\pm \binom{j}{k-l}\ne 0.$$
It is straightforward to show using the growth of $\binom{n}{k}$ in $k$ (if $k\le n/2$) that one term dominates the others if $l\ne k-l$ and $l\ne j-(k-l)$. Thus in these cases this (weighted) sum is nonzero. If $l= k-l$ or $l= j-(k-l)$, then either their sign is the same and hence the sum is nonzero, or their sign is different and hence they eliminate each other, hence the sum is 1 which is nonzero.

Summarizing, we get that $$\binom{k}{j}\pm \binom{k}{k-j,j-l,l}\pm \binom{k}{k-l,k-j,l-(k-j)}\ne 0$$ and hence
\begin{equation}\label{eq_fin}\sum_{i=1}^n \la_{i,k}\tla_{i,l}p_i^{j+1}q_i=0,\end{equation} which is equivalent to the inductive hypothesis for $j$ as we noted above.

Thus equation \eqref{eq_fin} holds for every $j=0, \dots, n-1$.
In matrix form this means that

\[
 \begin{pmatrix}
1& \ldots &1\\
 p_{1} & \ldots & p_{n} \\
 \vdots & \ddots & \vdots\\
 p_{1}^{n-1} & \ldots & p_{n}^{n-1}
 \end{pmatrix}
 \cdot
 \begin{pmatrix}
 p_1q_1\lambda_{1, k}\widetilde{\lambda}_{1, k}\\
 p_2q_2\lambda_{2, k}\widetilde{\lambda}_{2, l}\\
  \vdots \\
 p_nq_n\lambda_{n, k}\widetilde{\lambda}_{n, k}
 \end{pmatrix}
=
\begin{pmatrix}
 0\\
 0\\
 \vdots
 \\
 0
\end{pmatrix}.
 \]
Since this matrix is a Vandermonde type matrix with different $p_i$'s, as Case 1 and Case 2, this implies that, the only solution of this homogeneous linear system is the zero vector, i.e., $ p_iq_i\lambda_{i, k}\widetilde{\lambda}_{i, l}= 0$  for all $i=1, \ldots, n$. This contradicts to our assumption that there is some $i'$ for which $\lambda_{i', k}\ne 0$ and $\widetilde{\lambda}_{i', l}\ne 0$ (and $p_{i}\ne 0$, $q_i\ne 0$ by C(i)). This also finishes the proof of the theorem, thus the order of all derivations involved in \eqref{Eq_inner} is at most $n-1$.
\end{enumerate}
\end{proof}

\begin{remark}
 The upper bound appearing in the above theorem is sharp. To see this, let $n, N$ be positive integers,  $\lambda_{1}, \ldots, \lambda_{n}\in \mathbb{C}$ and let $f\colon \mathbb{F}\to \mathbb{C}$ be an additive function for which
 \[
  \sum_{i=1}^{n}\lambda_{i}f(x^{i})x^{N-i}=0
 \]
is fulfilled for all $x\in \mathbb{F}$. Then $f\in \mathscr{D}_{n-1}(\mathbb{R})$ if and only if $\lambda_{i}=(-1)^i \displaystyle\binom{n}{i}$ for all $i=1,\dots,n$, see \cite{Eba15, EbaRieSah, GseKisVin18}.
 \end{remark}

\begin{remark}\label{rem_inner_gen}
The proof of Theorem \ref{theorem_inner_degree} in all the cases is based on the fact that the matrix
\[
 \begin{pmatrix}
 \binom{p_{1}}{1}\binom{q_{1}}{1} & \ldots & \binom{p_{n}}{1}\binom{q_{n}}{1}\\
 \binom{p_{1}}{1,1}\binom{q_{1}}{1}+\binom{p_{1}}{1}\binom{q_{1,1}}{1} & \ldots & \binom{p_{n}}{1}\binom{q_{n}}{1, 1}+\binom{p_{n}}{1, 1}\binom{q_{n}}{1} \\
 \vdots & \ddots & \vdots\\
  \binom{p_{1}}{\underbrace{1, \ldots, 1}_{n-1}}\binom{q_{1}}{\underbrace{1, \ldots, 1}_{n -1}} & \ldots & \binom{p_{n}}{\underbrace{1, \ldots, 1}_{n-1}}\binom{q_{n}}{\underbrace{1, \ldots, 1}_{n-1}}
 \end{pmatrix}
\]
has rank $n$, though it is far from being trivial to find the proper sub-matrix, which verifies that.

On the other hand, the situation is much more complicated, if the maximal order $k$ of $D_i$, and  the maximal order $l$ of $\widetilde{D_i}$ is not uniquely determined. Namely, if there are several different pairs $(k_i, l_i)$ so that $k_i+l_i=K$, where $K$ is constant and $k_i$ is the maximum order of $D_i$, $l_i$ is the maximal order of $\widetilde{D_i}$. Then the equations first can only be determined for subsets of the index set, which satisfy some nontrivial relations. In this case the first task is to show that the problem can be formalized separately for the index sets, which seems a very hard problem in full generality. In this case our method can be applied.
\end{remark}

Theorem \ref{theorem_inner_degree} and Remark \ref{rem_inner_gen} motivates our conjecture, that we verified for $n\le 4$.
\begin{conj}\label{conj_inner}
Let $n$ be a positive integer, $\mathbb{F}\subset \mathbb{C}$ be a field and
 $p_{1}, \ldots, p_{n}, q_{1}, \allowbreak \ldots, q_{n}$ be fixed positive integers fulfilling conditions C(i)--C(iii).
 Assume that the additive functions $f_{1}, \ldots, f_{n}, g_{1}, \ldots, g_{n}\colon \mathbb{F}\to \mathbb{C}$ satisfy equation \eqref{Eq_inner}. Then every solution is a generalized exponential polynomial function of degree at most $n-1$ on $\mathbb{F}^{\times}$.
 In particular, if
 \begin{equation}\label{eq_inred3}
  f_{i}(x)= D_i(x)
  \qquad
  \text{and}
  \qquad
  g_{i}(x)= \widetilde{D}_i(x)
  \qquad
  \left(x\in \mathbb{F}^{\times}\right)
 \end{equation}
for each $i=1, \ldots, n$, then the order of $D_i, \widetilde{D}_i$ is  at most $n-1$.
\end{conj}
This conjecture leads to the following more general open question.
\begin{opq}\label{op1}
Is it true that every nonzero additive, irreducible solutions $f_1,\dots, f_n$ of
$$P(f_1(x^{p_1}),\dots,f_n(x^{p_n}))=0, ~~~~\textrm{with }~~~ P(0,\dots, 0)=0$$
are derivations of order at most $n-1$ and the identity function up to a homomorphism, if $P:\mathbb{C}\to \mathbb{C}$ is polynomial and $p_1\dots, p_n$ are distinct positive integers?
\end{opq}

Finally we highlight some important special cases when Theorem \ref{theorem_inner_degree} gives the proper bound of the order of the derivations.
\begin{cor}\label{cor_inner_spec}
Let $n$ be a positive integer, $\mathbb{F}\subset \mathbb{C}$ be a field and
 $p_{1}, \ldots, p_{n}, q_{1}, \allowbreak \ldots, q_{n}$ be fixed positive integers fulfilling conditions C(i)--C(iii).
 Assume that the additive functions $f_{1}, \ldots, f_{n}, g_{1}, \ldots, g_{n}\colon \mathbb{F}\to \mathbb{C}$ satisfy equation \eqref{Eq_inner} as an irreducible solution.
 Then $f_i\sim D_i$ and $g_i\sim\widetilde{D_i}$, where $D_i$ and $\widetilde{D_i}$ are higher order derivations.
 Assume further that one of the following holds.
 \begin{enumerate}[(A)]
     \item  All $D_i$ have the same order. This is the case, when $f_i(x)=c_if(x)\, (x\in \mathbb{F})$ for some nonzero constants $c_i\in \mathbb{C}$, $i\in \{1, \dots, n\}$.
     \item All $\widetilde{D_i}$ have the same order. This is the case, when $g_i(x)=c_ig(x)\, (x\in \mathbb{F})$ for some nonzero constants $c_i\in \mathbb{C}$, $i\in \{1, \dots, n\}$.
     \item $f_i=c_i\cdot g_i$ for all $i\in \{1, \dots, n\}$ with some nonzero constants $c_i\in \mathbb{C}$, $i=1, \ldots, n$.
 \end{enumerate}
 Then the order $D_i$ and $\widetilde{D_i}$ is at most $n-1$.
\end{cor}
\noindent
{\it Proof.}
Theorem \ref{theorem_inner} implies that every solution $f_i$ (resp. $g_i$) is an exponential polynomial of the form $P_i\cdot m$ (resp. $Q_i\cdot m$), which means that $f_i\sim D_i$ and $g_i\sim \widetilde{D_i}$ for some derivations $D_i$ and $\widetilde{D_i}$.
\begin{enumerate}[(A)]
    \item Let $k$ denote the order of $D_i$, and let $l$ be the maximal order of $\widetilde{D_i}$. Now we are in the position to apply Theorem \ref{theorem_inner_degree}.
    \item Similar to (A), since in case of equation \eqref{Eq_inner}, the role of the parameters $p_{1}, \ldots, p_{n}$ and $q_{1}, \ldots, q_{n}$ is symmetric.
    \item As the maximal degree of $f_i$ is the same as the maximal degree of $g_i$ and it is taken for the same index we can apply Theorem \ref{theorem_inner_degree}.
   \qed
\end{enumerate}

\subsubsection*{Special cases of equation \eqref{Eq_inner}}

In this subsection we consider special cases of equation \eqref{Eq_inner}.
All the equations we consider here are of the form
\[
 f_{1}(x^{p_{1}})g_{1}(x^{q_{1}})+f_{2}(x^{p_{2}})g_{2}(x^{q_{2}})=0
 \qquad
 \left(x\in \mathbb{F}\right).
\]
Here $f_{1}, f_{2}, g_{1}, g_{2}\colon \mathbb{F}\to \mathbb{C}$ denote the unknown additive functions and the parameters $p_{1}, p_{2}, q_{1}, q_{2}$ fulfill conditions C(i)--C(iii). Due to the results of the previous section, we get that
\[
\begin{array}{rcl}
 f_{i}(x) &\sim& \lambda_{i, 0}x+\lambda_{i, 1}d_{i}(x) \\
 g_{i}(x) &\sim& \mu_{i, 0}x+\mu_{i, 1}\widetilde{d_{i}}(x)
 \end{array}
 \qquad
 \left(x\in \mathbb{F}, i=1, 2\right)
\]
with appropriate complex constants $\lambda_{i, j}, \mu_{i, j}$ ($i=1, 2$, $j=0, 1$) and derivations $d_{i}, \tilde{d_{i}}\colon \mathbb{F}\to \mathbb{C}$ ($i=1, 2$).

\begin{cor}
 Let $N$ be a positive integer, $\mathbb{F}\subset \mathbb{C}$ be a field and  $p, q$ be different positive integers (strictly) less than $N$ and assume that $q\neq N-p$. If the additive functions $f, g\colon \mathbb{F}\to \mathbb{C}$ satisfy
 \[
  f(x^{p})f(x^{N-p})=g(x^{q})g(x^{N-q})
  \qquad
  \left(x\in \mathbb{F}\right),
 \]
then
\begin{enumerate}[A)]
 \item either there exists a homomorphism $\varphi\colon \mathbb{C}\to \mathbb{C}$, a derivation $d\colon \mathbb{F}\to \mathbb{C}$ such that
 \[
  f(x)= \varphi(d(x))
  \qquad
  \text{and}
  \qquad
  g(x)= \alpha\varphi(d(x))
  \qquad
  \left(x\in \mathbb{F}\right),
 \]
where $\alpha=\dfrac{p(N-p)}{q(N-q)}$,
\item or there exists a homomorphism $\varphi\colon \mathbb{F}\to \mathbb{C}$ such that
\[
f(x)= f(1)\cdot \varphi(x)
\qquad
g(x)= \pm f(1)\cdot \varphi(x)
\qquad
\left(x\in \mathbb{F}\right).
\]
\end{enumerate}
\end{cor}

\begin{cor}
 Let $N$ be a positive integer, $\mathbb{F}\subset \mathbb{C}$ be a field and  $p, q$ be different positive integers (strictly) less than $N$ and assume that $q\neq N-p$. If the additive functions $f, g\colon \mathbb{F}\to \mathbb{C}$ satisfy
 \[
  f(x^{p})g(x^{N-p})= \kappa f(x^{q})g(x^{N-q})
  \qquad
  \left(x\in \mathbb{F}\right).
 \]
Then \begin{enumerate}[A)]
 \item if $\kappa\notin \left\{ 1, \dfrac{p(N-p)}{q(N-q)} \right\}$, then $f$ is identically zero,
 \item if $\kappa= 1$, then the only possibility is that
  \[
  f(x)= f(1)\cdot \psi(x)
  \qquad
  \text{and}
  \qquad
  g(x)= f(1)\cdot \psi(x)
  \qquad
  \left(x\in \mathbb{F}\right),
 \]
 where $\psi\colon \mathbb{F}\to \mathbb{C}$ is a homomorphism,
 \item if $\kappa=\dfrac{p(N-p)}{q(N-q)}$, then there exists a homomorphism $\varphi\colon \mathbb{C}\to \mathbb{C}$ and derivations $d_1, d_2 \colon \mathbb{F}\to \mathbb{C}$ such that
  \[
   f(x)=  \varphi(d_1(x))
   \qquad
   \text{and}
   \qquad
   g(x)= \varphi(d_2(x))
   \qquad
   \left(x\in \mathbb{F}\right).
  \]
\end{enumerate}
\end{cor}
\noindent
Both results implies that the nonzero additive solutions of equation
\[
  f(x^{p})f(x^{N-p})=\kappa f(x^{q})f(x^{N-q})
  \qquad
  \left(x\in \mathbb{F}\right).
 \]
are derivations of order 1 ($\kappa = \dfrac{p(N-p)}{q(N-q)}$) or the identity function ($\kappa=1$) up to a homomorphism.

\begin{ackn}
The research of E.~Gselmann has been supported by project no.~K134191 that has been implemented
with the support provided by the National Research, Development and Innovation Fund of Hungary,
financed under the K\_20 funding scheme.\\
The research of G.~Kiss has been supported by projects no.~K124749  and no.~K142993 of the National Research, Development and Innovation Fund of Hungary. The author have supported by Bolyai János Research Fellowship of the Hungarian Academy of Sciences and ÚNKP-22-5 New National Excellence Program of the Ministry for Culture and
Innovation.
\end{ackn}

\bibliographystyle{plain}

\end{document}